\documentclass[11pt,a4paper]{article}
\usepackage{fullpage}
\usepackage{amsmath,amssymb}
\usepackage{enumitem}
\usepackage{pgf,tikz} 
\usepackage{pgfplots}

\usepackage[amsthm,thmmarks]{ntheorem}
\numberwithin{equation}{section}
\theoremstyle{change} 

\newtheorem{theorem}{Theorem}[section]
\newtheorem{lemma}[theorem]{Lemma}
\newtheorem{definition}[theorem]{Definition}%[section]

\newtheorem{corollary}[theorem]{Corollary}

\theorembodyfont{\normalfont}
\theoremstyle{change}
\theorembodyfont{\upshape}
\theoremsymbol{\ensuremath{\ast}\theoremseparator{}}

\newtheorem{remark}[theorem]{Remark}

\newtheorem{example}[theorem]{Example}

\DeclareMathOperator{\id}{id}
\newcommand{\R}{\ensuremath{\mathbb{R}}}
\newcommand{\N}{\ensuremath{\mathbb{N}}}
\newcommand{\dv}[1]{\mathrm{d}#1}
\DeclareMathOperator{\diag}{\mathrm{diag}}

\newcommand{\Barg}[1]{\Bigl(#1\Bigr)}
\newcommand{\nset}[1]{\{#1\}}

\newcommand{\Bnorm}[1]{\Bigl\|#1\Bigr\|}

\newcommand{\Bbar}[1]{\Bigl|#1\Bigr|}

\newcommand{\bsqb}[1]{\bigl[#1\bigr]}
\newcommand{\Bsqb}[1]{\Bigl[#1\Bigr]}
\newcommand{\hide}[1]{{}}

\usepackage{color} 
\usepackage{pgf}

\begin{document}

\begin{center}
{\LARGE Stress-modulated growth in the presence of nutrients -- \\[.1in] existence and uniqueness in one spatial dimension\\[.3in]} \today \\[.3in]
 Kira Bangert and Georg Dolzmann\footnote{Fakult\"at f\"ur Mathematik, Universit\"atsstra{\ss}e 31, 93053 Regensburg, Deutschland, georg.dolzmann@ur.de (corresponding author)}
 \vspace*{.2in}
\end{center}


\medskip

\begin{abstract}
Existence and uniqueness of solutions for a class of models for stress-modulated growth is proven in one spatial dimension. The model features the multiplicative decomposition of the deformation gradient $F$ into an elastic part $F_e$ and a growth-related part $G$. After the transformation due to the growth process, governed by $G$, an elastic deformation described by $F_e$ is applied in order to restore the Dirichlet boundary conditions and therefore the current configuration might be stressed with a stress tensor $S$. The growth of the material at each point in the reference configuration is given by an ordinary differential equation for which the right-hand side may depend on the stress $S$ and the pull-back of a nutrient concentration in the current configuration, leading to a coupled system of ordinary differential equations.
\end{abstract}

\section{Introduction}
In the past twenty years, the broad area of biomechanics has gained an increasing interest in the mathematical community in view of the many challenging questions that are related to the analysis of the models that have been proposed. One approach is inspired by models for crystal plasticity in terms of modern continuum mechanics in the framework of nonlinear theories which were developed, for example, in~\cite{Kroner1960,Lee1969}. They are based on a multiplicative decomposition of the total deformation gradient into an elastic part and a plastic part. Here, the plastic part encodes the permanent deformation of the initial reference configuration due to plastic deformation and the elastic part describes the subsequent elastic deformation of the plastically deformed reference configuration onto the currently observed configuration. These ideas have been adapted for growth processes, for example in~\cite{LubardaHoger2002,MenzelKuhl2012,RodriguezHogerMcCulloch1994,SkalakEtAlJTheorBiol1982,TaberApplMechRev1995}, where the initial deformation of the reference configuration due to plastic deformation is replaced by an initial deformation due to the growth processes in the material. A theory based on a split of the process into two distinct subprocesses, namely the volumetric, stress-free growth and the subsequent mechanical response, has been developed in \cite{AmbrosiEtAl2011,AmbrosiGuillouContMechTD2007,AmbrosiMollicaIntJEngrgSci2002,GorielyBook,GorielyMoultonOUP2010}, see also the references therein. One particular feature of these models is that they predict residual stress and it is, for example, an open problem to calculate the opening angle produced by a given residual stress field in a soft biological tissue, see, e.g., \cite{HuangOgdenPentaJElast2021,VanDykeHoger2002OpeningAngle}. A different approach to problems related to growth processes in the presence of an active environment like a bood vessel is based on a free boundary fluid-structure interaction problem and was proposed in~\cite{YangEtAlJMathBio2016BloodVessels}. Very recently, significant analytical progress was achieved in~\cite{AbelsLiuArXive2021,AbelsLiuArXive2022}. Ideas to understand mechanical properties for plant tissues in the convext of homogenization were developed in~\cite{PiatnitskiPtashnyk2020}.

In this article, we focus on one aspect in this broad context, namely stress modulated growth, that is, the interplay of growth processes with elastic stress. Our main purpose is to present a complete and self-contained analytical treatment in the sense of existence and uniqueness of solutions for one particular class of models, and therefore we concentrate on a one-dimensional situation. Major obstacles that must be addressed in order to extend our results to a higher-dimensional case, are commented on in Section~\ref{sec:conclusions}.

\medskip

\textit{The multiplicative decomposition of the deformation gradient:} The model we are interested in is closely related to the one-dimensional models in \cite[Ch. 4]{GorielyBook}. We consider a growing body in one dimension, that is, a material that is constrained to a straight line and that can undergo a change in its conformation by a change in density or length, but not by effects like buckling, bending, or twist which were considered in~\cite{BressanPalladinoShen2017}. Suppose that the body is given in its initial reference configuration at time $t=0$ by the interval $I_0=(0,L_0)\subset \R$ with $L_0>0$. As it is customary in morphoelasticity, we introduce a natural or virtual configuration of the material after unconstrained growth at time $t$, $I_v(t)=(0,L(t))$, and the current or deformed configuration, $I_c(t) = (0,\ell(t))$. Since the configuration $I_v(t)$ can be realized by an interval, we refer to it in this particular one-dimensional situation as the natural configuration of the material at time $t$ but keep using the index $v$. The goal is to determine, for a given time horizon $T$, a deformation $y:Q_T=I_0\times (0,T) \to \R$ such that $y(I_0,t)$ is the current configuration of the growing body at time $t$. Here, we employ the hypothesis of a multiplicative structure, namely that the deformation gradient $\partial_X y(X,t)$ is given by a product of an elastic part $F_e(X,t)$ and a growth related part $G(X,t)$ such that the equation 
\begin{align}\label{multiplicative}
\partial_X y(X,t) = F_e(X,t) G(X,t)
\end{align}
is satisfied. It is a very special aspect of the one-dimensional situation that $y$ is determined from this equation and suitable boundary conditions in $X$ as soon as $F_e$ and $G$ have been found. Before giving the precise definition, we discuss the interpretation of $F_e$ and $G$. The growth related factor $G$ in the deformation gradient prescribes the local change of an infinitesimal material element in the body due to unconstrained growth, that is, $G$ is the deformation gradient of the deformation from the reference configuration into the natural deformation. Again, due to the special structure in one dimension, there exists a potential $g$ in the variable $X$, the deformation due to growth, with $g(0,t)=0$ and $\partial_X g(X,t) = G(X,t)$ and it is required that $g(\cdot, t)$ is a bijection for all $t$ fixed. This is, for example, the case if $G>0$, an assumption which will be satisfied in our model. 

As in \cite[Sec. 4.2]{GorielyBook}, we assume that the growing body is confined by two rigid plates which leads to the boundary conditions $y(0,t)=0$ and $y(L_0,t)=\ell_0$ for all $t\in (0,T)$. In the following we assume that $L_0=\ell_0$ but we keep the distinct symbols to indicate the distinct configurations, namely the reference configuration and the currently observed configuration. This identification is not necessary but convenient since it leads to a stress-free initial configuration.

\medskip
\textit{Elastic stress after growth due to boundary conditions:}
The boundary conditions for the growing material will, in general, be violated after unrestricted growth which leads to the natural configuration $I_v(t) = (g(0,t), g(L_0,t))$. Therefore, a superimposed elastic deformation $\phi=\phi(z, t)$ of the natural configuration $I_v(t)$ is needed in order to restore the boundary values and this elastic deformation is obtained as a minimizer of a hyperelastic variational integral defined in the natural configuration. This requires an assumption on the elastic properties of the material after growth. Our model is based on the assumption that the material point $X$ in the reference configuration and the material point $g(X,t)$ in the natural configuration have the same elastic properties. Hence, it suffices to define a stored energy density $W:[0,L_0]\times (0,\infty)\to \R$ in the reference configuration and this stored energy density induces the stored energy density $W_G$ in the natural configuration via $W_G(\cdot,p) = W(g^{-1}(\cdot,t),p)$. Consequently, the elastic deformation $\phi(\cdot, t)$ is a minimizer of a hyperelastic variational integral
\begin{align}\label{VariationalIntegral}
\int_{g(0,t)}^{g(L_0,t)} W_G(z, \partial_z\phi(z,t))\dv{z} =  \int_{g(0,t)}^{g(L_0,t)} W(g^{-1}(z,t), \partial_z\phi(z,t))\dv{z}
\end{align}
subject to the boundary conditions $\phi(g(0,t),t)=0$ and $\phi(g(L_0,t),t)=\ell_0$. The matrix $F_e$ in~\eqref{multiplicative}, the elastic contribution to the total deformation gradient, is given by $F_e(X,t)=(\partial_z \phi)(g(X,t),t)$. The Euler-Lagrange equation for the variational problem implies that the total stress
\begin{align}\label{defstress}
 S(t)=(\partial_p W)(g^{-1}(z,t),\partial_z\phi(z,t))
\end{align}
in the system is constant in $z$. The deformation due to growth, $g(\cdot, t):[0,L_0]\to \R$,  $X\mapsto g(X,t)$ and the deformation due to elastic forces, $\phi(\cdot, t):[g(0),g(L_0)]\to \R$, determine the total deformation $y(\cdot ,t) = \phi(g(\cdot, t),t)$.

\medskip

\textit{The reaction-diffusion equation for the nutrients:}
The model is completed by a reaction-diffusion equation for nutrients in the current configuration of the form 
\begin{align}\label{reactiondiffusion}
 -\partial_x (D(x,t) \partial_x n(x,t)) = -\beta(x,t) n(x,t)\quad \text{ for }x\in (0, \ell_0),\, t\in (0,T)
\end{align}
subject to boundary conditions $n(0,t) = n_L$ and $n(\ell_0,t)=n_R$ with $n_L$, $n_R>0$ where the diffusion constant $D$ and the absorption rate $\beta$ depend on space and time. The specific choice of the constants at time $t$ in the current deformation is part of the model and in order to illustrate a class of models which can be treated by our theory, the following assumptions are made for the two stages in the deformation process, first the change in the reference configuration due to growth or absorption and then the change from the natural configuration to the observed configuration due to elastic deformation. 

For the growth process it is assumed that the diffusion coefficient and the absorption rate of the material generated by growth at time $t$ in $z=g(X,t)$ in the natural configuration $I_v(t)$ are given by the diffusion coefficient and the growth rate of material in the reference configuration at time $t=0$ in $X$. Thus, the diffusion constant $D_v$ and the reaction rate $\beta_v$ in the natural configuration after growth according to $G(\cdot, t)\in \mathcal{L}^\infty([0,L_0])$ are given, at $z= g(X,t)$ and time $t\in [0,T]$, by 
\begin{align*}
 D_v(z,t) = (D_0 \circ g^{-1})(z,t)\,,\quad \beta_v(z,t) = (\beta_0\circ g^{-1})(z,t)\,.
\end{align*}
The transformation of the diffusion constant and the reaction rate during the elastic deformation of the material is based on the assumption that the diffusion in the current configuration corresponds to the diffusion in the natural configuration via the change of variables given by the elastic deformation $\phi$. That is, if $n_v(\cdot, t)$ is the solution of the reaction-diffusion equation on $I_v(t)$ with coefficients $D_v(\cdot, t)$ and $\beta_v(\cdot, t)$ and if $n(\cdot, t)$ is the solution of the reaction-diffusion equation on $I_c(t)$ with coefficients $D(\cdot, t)$ and $\beta(\cdot, t)$, then, at $x=\phi(z)$, the equation $n(x, t) = n_v(z, t)$ holds.  The corresponding equations for $n_v(\cdot, t)$ and $n(\cdot, t)$ are given for all test functions $\psi_v \in H_0^1(g(0,t), g(L_0,t))$ by
\begin{align*}
 \int_{g(0,t)}^{g(L_0,t)} D_v(z,t) \partial_z n_v(z, t) \partial_z \psi_v(z) \dv{z} = -\int_{g(0,t)}^{g(L_0,t)} \beta_v(z,t) n_v(z,t) \psi_v(z)\dv{z}
\end{align*}
and for all test functions $\psi\in H_0^{1}(0,\ell_0)$ by
\begin{align*}
\int_{0}^{\ell_0} D(x,t) \partial_x n(x,t) \partial_x\psi(x)\dv{x} = - \int_0^{\ell_0}\beta(x,t) n(x,t) \psi(x)\dv{x}\,,
\end{align*}
respectively.

By a change of variables in the weak equation for $n(\cdot, t)$, we find, with $0=\phi(g(0,t),t)$ and $\ell_0 = \phi(g(L_0,t),t)$, that
\begin{align*}
&\int_{g(0,t)}^{g(L_0,t)} D(\phi(z,t),t) \partial_x n(\phi(z,t),t) \partial_x \psi(\phi(z,t))\partial_z\phi(z,t)\dv{z}& \\ &\qquad = - \int_{g(0,t)}^{g(L_0,t)}\beta(\phi(z,t),t) n(\phi(z,t),t) \psi(\phi(z,t))\partial_z\phi(z,t)\dv{z}\,.
\end{align*}
Since $\phi$ is a diffeomorphism, $\widetilde{\psi}(\cdot,t)=(\psi\circ \phi)(\cdot,t)\in H_0^1(g(0,t),g(L_0,t))$ is an admissible test function and the weak equation after the change of variables can be rewritten as 
\begin{align*}
&\int_{g(0,t)}^{g(L_0,t)} D(\phi(z,t),t) \partial_z[n(\phi(z,t),t)] \partial_z\widetilde{\psi}(z,t)\frac{1}{\partial_z\phi(z)}\dv{z}& \\ &\qquad = - \int_{g(0,t)}^{g(L_0,t)} \beta(\phi(z,t),t) n(\phi(z,t),t) \widetilde{\psi}(z,t)\partial_z\phi(z,t)\dv{z} \,.
\end{align*}
The requirement that this equation corresponds to the reaction-diffusion equation in the natural configuration leads to the assumptions 
\begin{align*}
 D_v(z,t) = D_0(g^{-1}(z,t)) = \frac{D(\phi(z,t),t)}{\partial_z\phi(z,t)} \quad \Leftrightarrow \quad D(x,t) = D_0(g^{-1}(\phi^{-1}(x,t),t)) \partial_z\phi(\phi^{-1}(x,t),t)
\end{align*}
and 
\begin{align*}
 \beta_v(z,t) = \beta_0(g^{-1}(z,t)) = \beta(\phi(z,t),t) \partial_z\phi(z,t) \quad \Leftrightarrow \quad \beta(x,t) = \frac{\beta_0(g^{-1}(\phi^{-1}(x,t),t)}{\partial_z\phi(\phi^{-1}(x,t),t)}\,.
\end{align*}
If the material experiences strong growth, the corresponding elastic deformation is a strong compression and the associated diffusion constant is small. 
In order to have the usual elliptic regularity estimates at our disposal, we assume in the following that $D_0$, $\beta_0\in W^{1,\infty}([0,L_0])$\hide{, $\beta_0\in \mathcal{L}^\infty([0,L_0])$} and that there exist constants $D_{min}$, $D_{max}$, $\beta_{min}$, $\beta_{max}\in (0,\infty)$ with
\begin{align}\label{bounds:DandBeta}
 \forall X\in [0,L_0]\colon D_0(X) \in [D_{min}, D_{max}]\,,\quad \beta_0(X)\in [\beta_{min}, \beta_{max}]\,.
\end{align}

\bigskip

\textit{General assumptions for the growth dynamics:}
The foregoing discussion shows that the deformation gradient $G$ of the growth map $g$ is the basic variable: For a given $G$ one finds the growth map $g$ by integration, the natural configuration $I_v$, the subsequent elastic deformation $\phi:I_v\to I_c$ as a minimizer of the hyperelastic variational problem and the current diffusion coefficient $D$ and absorption rate $\beta$ which determine the nutrient concentration $n$. The deformation $y$ from the reference configuration onto the deformed configuration can be used to define the nutrient configuration in the reference configuration as $N=n\circ y$. It remains to state an evolution equation for $G$. Here, we follow the approach that $G$ satisfies at each point $X\in [0,L_0]$ an ordinary differential equation of the form
\begin{align}\label{growthequation}
 \dot{G}(X,t) = \mathcal{G}(G(t,X), S(t), N(\cdot,t),X)\,,\quad G(X,0)=1\,,
\end{align}
that is, the rate of growth depends on the stress $S(t)$ and the nutrient concentration $N(\cdot,t)$ in the system at time $t$. Since the model prescribes the growth at every point $X\in [0,L_0]$ it is natural to use the space $\mathcal{L}^\infty([0,L_0])$ of all measurable and bounded functions together with the supremum norm as the basic function space for the variable $G$. Moreover, the evolution equations at each point $X$ are coupled to the state of the system at all points in $[0,L_0]$ since both the stress $S$ in the system as well as the referential description of the nutrient concentration $N$ are nonlocal functions. Therefore we interpret the evolution~\eqref{growthequation} as an evolution of the system in the Banach space $\mathcal{L}^\infty([0,L_0])$ on a time interval $[0,T]$ with $T>0$.

\begin{definition}\label{def:solutionG}
 A function $G\in C^1([0,T];\mathcal{L}^\infty([0,L_0]))$ determines a solution of the one-di\-men\-sional growth problem with initial condition $G_0\in \mathcal{L}^\infty([0,L_0])$ with $G_0>0$ everywhere if $G$ is a solution of the ordinary differential equation 
 \begin{align*}
  \dot{G} = \widehat{\mathcal{G}}(G)\,,\quad G(0) = G_0
 \end{align*}
in the Banach space $\mathcal{L}^\infty([0,L_0])$. Here, for $X\in [0,L_0]$ and $G\in \mathcal{L}^\infty([0,L_0])$,
\begin{align*}
 \widehat{\mathcal{G}}(G)(X) = \mathcal{G}(G(X), S(G), N(G)(X),X)
\end{align*} 
with a suitable function $\mathcal{G}$ (see Section~\ref{sec:existence} for the precise assumptions) where $S=S(G)\in \R$ and $N=N(G)=n(G)\circ y(G)\in H^1(0,L_0)$ denote the elastic stress due to the elastic deformation and the nutrient concentration in the reference configuration according to~\eqref{defstress} and~\eqref{reactiondiffusion}.
\end{definition}

The decomposition~\eqref{multiplicative} is thus given by 
\begin{align*}
 \partial_X y(X,t) = F_e(X,t) G(X,t) = (\partial_z \phi)(g(X,t),t) \partial_X g(X,t)\,.
\end{align*}
Since we consider $G$ to be the basic unknown in the system and not $g$, the anti-derivative of $G$ in $X$, we define $y$ as the anti-derivative of $F_e(X,t) G(X,t)$ in $X$ and interpret the decomposition for $\partial_X y$ as a pointwise identity.

We illustrate this concept with a few examples and begin with the situation without any feedback of the growth in the body onto the local growth via the stress or the nutrient concentration in the material. The only dependence of the growth on the material point $X$ is given by a growth coefficient $\gamma(X)$ which we assume to be piecewise constant on two subintervals of $[0,L_0]$. 

\begin{example}[pure growth]\label{ex:puregrowth}
Suppose that $X_I\in (0,L_0)$, $\gamma_0$, $\gamma_1\in \R$, $\gamma_0, \gamma_1>0$, that $\gamma\in \mathcal{L}^\infty([0,L_0])$ is given by $\gamma(X) = \chi_{[0,X_I]}\gamma_0 + \chi_{(X_I,L_0]}\gamma_1$ and that 
\begin{align*}
 \widehat{\mathcal{G}}:\mathcal{L}^\infty([0,L_0]) \to \mathcal{L}^\infty([0,L_0])\,,\quad G\mapsto  \widehat{\mathcal{G}}(G)\text{ with } \widehat{\mathcal{G}}(G)(X) = \mathcal{G}(G(X),X) = \gamma(X) G(X)\,.
\end{align*}
Then, the evolution of $G$ is determined from the initial condition $G(0)=G_0\in \mathcal{L}^\infty([0,L_0])$ via $G(t,X) = G_0(X) \exp(\gamma(X)t)$. The function $g$ is Lipschitz continuous in $X$ and at the beginning of the growth process modeled by the initial condition $G_0(X)=1$ for all $X\in [0,L_0]$ given by 
\begin{align*}
 g(X,t) & = \int_0^X G(U,t)\dv{U} =  \int_0^X G_0(U) \exp(\gamma(U)t)\dv{U} \\[.1in] & = \left\{\begin{array}{cl} X \exp(\gamma_0 t) & \text{ if }X\leq X_I\,,\\[.05in] X_I \exp(\gamma_0 t) + (X-X_I) \exp(\gamma_1 t) & \text{ if }X > X_I\,. \end{array} \right.
\end{align*}
In fact, for all $t$ the map $g(\cdot, t)$ is bijective and bi-Lipschitz.
\end{example}

Suppose that the elastic properties of the material in the reference configuration at a point $X$ are given by a stored energy density $W(X, \cdot)$. Since we assume that the elastic properties do not change during the growth process, the stored energy density of the material after growth in the natural configuration $[g(0,t),g(L_0,t)]$ is given by 
\begin{align*}
 W_G: [g(0,t),g(L_0,t)]\times (0,\infty)\to \R\,,\quad (z, p) \mapsto W_G(z,p) = W(g^{-1}(z,t), p)\,.
\end{align*}
To be specific, we postpone the discussion of the corresponding variational problem to Section~\ref{sec:hyperelastic} and illustrate the ideas along the lines of Example~\ref{ex:puregrowth}. Assume that $\kappa_0$, $\kappa_1\in (0,\infty)$ and that $W_0:(0,\infty)\to \R$ is a smooth and strictly convex function with $W_0(p)\to \infty$ as $p\to 0$ and $W_0(p)/p\to \infty$ as $p\to \infty$. Suppose that $\kappa:[0,L_0]\to \R$ is given by $\kappa = \chi_{[0,X_I]}\kappa_0 + \chi_{(X_I,L_0]}\kappa_1$ and that $W(X,p) = \kappa(X) W_0(p)$. The Euler-Lagrange equation for the minimizer $\phi$ of the variational integral~\eqref{VariationalIntegral} implies that the stress 
\begin{align*}
 S(t) = (\partial_p W_G)(z, \partial_z\phi(z,t)) = \kappa(g^{-1}(z,t)) (\partial_p W_0)(\partial_z\phi(z,t))
\end{align*}
is constant and therefore $\partial_z\phi(\cdot,t)$ is constant on $[g(0,t),g(X_I,t)]$ and $(g(X_I,t), g(L_0,t)]$, respectively, and an explicit formula for $\phi(\cdot, t)$ can be obtained from the boundary conditions and the continuity of the stress. The composition $y(X,t)=\phi(g(X,t),t)$ with $\partial_X y(X,t) = (\partial_X \phi)(g(X,t),t) \partial_X g(X,t) = (\partial_X \phi)(g(X,t),t) G(t,X)$ determines the deformation of the reference configuration onto the deformed configuration. In Figure~\ref{fig:RefVirt} a sketch of the siutation with a given growth tensor $G$ as a graph over the reference configuration and the elastic deformation as a graph over the natural configuration is shown.

It is illustrative to rewrite the equation for $S$ for elastic deformations with $\partial_z\phi$ small. In this case, we may assume in the sense of a local Taylor series expansion that $W_0(p)=(1/2)(p-1)^2$ and we obtain an explicit representation for $\partial_z\phi$, 
\begin{align*}
 S(t) = \kappa(g^{-1}(z,t)) (\partial_z\phi(z,t)-1)\quad \Leftrightarrow\quad \partial_z\phi(z,t) = \frac{S(t)}{\kappa(g^{-1}(z,t))}+1\,.
\end{align*}
Note that $g$ is the anti-derivative of a measurable and bounded function and thus Lipschitz continuous. Therefore, the regularity of $\partial_z\phi$ depends on the properties of $\kappa$. If $\kappa$ is merely positive, measurable and bounded, the same is true for $\partial_z\phi$. However, if $\kappa$ is positive and at least Lipschitz continuous, $\partial_z\phi$ is Lipschitz continuous as a concatenation of Lipschitz continuous functions. The Lipschitz constant depends on global bounds on $S$, on $\kappa$, on the Lipschitz constant of $\kappa$, and on bounds for $\partial_X g=G$, which appear in the calculation of the Lipschitz constant of $g^{-1}$, only.

\begin{example}[coupling via stress]\label{ex:stresscoupling}
Consider the situation in Example~\ref{ex:puregrowth} and assume that $\mu\in W^{1,\infty}(\R)$ is bounded and increasing. Suppose that the evolution is given by
\begin{align*}
 \widehat{\mathcal{G}}:\mathcal{L}^\infty([0,L_0]) \to \mathcal{L}^\infty([0,L_0])\,,\quad G\mapsto  \widehat{\mathcal{G}}(G)\text{ with } \widehat{\mathcal{G}}(G)(X) = \gamma(X)\mu(S(G)) G(X)
\end{align*}
where $S(G)$ denotes the stress in the body after stress-free growth given by $G$ and subsequent elastic deformation. The fact that $\mu$ is increasing reflects the assumption that tensile stress favors growth within certain limits and that compressive stress inhibits growth and may lead to absorption. Of course, other assumptions can be made.
A typical example is the family of functions 
\begin{align*}
 \mu(S) = a \arctan(S-b) + c\,,\quad a,b,c\in \R,\, a>0\,.
\end{align*}
If $b$ is interpreted as a homeostatic stress, the linearization about $b$ coincides with the model in~\cite[Sec. 4.4.3]{GorielyBook}.
 
As in Example~\ref{ex:puregrowth}, once the map $G(\cdot, t)$ is known at time $t$, we can calculate the elastic deformation which is piecewise affine and continuous and determines the stress in the body. Thus, the right-hand side in~\eqref{growthequation} (no dependence on $N$) is well defined and we can solve the ordinary differential equation in the Banach space $\mathcal{L}^\infty([0,L_0])$. This can be achieved based on Picard-Lindel\"of's theorem on existence and uniqueness if the right-hand side is a Lipschitz continuous function in $G$. Therefore, the key step in the proof of the existence theorem is the verification of the Lipschitz continuity of all relevant terms on the right-hand side of the ordinary differential equation. 
\end{example}

\begin{example}[coupling via stress and nutrients]\label{ex:fullcoupling}
 Finally, we consider the situation in Example~\ref{ex:stresscoupling} and assume that there exists a Lipschitz continuous function $\eta:\R\to \R$ such that  
\begin{align*}
 \widehat{\mathcal{G}}:\mathcal{L}^\infty([0,L_0]) \to \mathcal{L}^\infty([0,L_0])\,,\quad G\mapsto  \widehat{\mathcal{G}}(G)\text{ with } \widehat{\mathcal{G}}(G)(X) = \gamma(X)\mu(S(G))\eta(N(G)(X)) G(X)\,.
\end{align*}
The function $\eta=\id_{\R}$ is an admissible choice since the solutions of the reaction-diffusion equation are bounded and nonnegative. Here, $N=n\circ y$ denotes the referential description of the nutrient concentration. In Figure~\ref{fig:Current} a sketch of the nutrient concentration as a graph over the current configuration is shown. 
\end{example}

In order to ensure the Lipschitz continuity of the right-hand side in the evolution equation~\eqref{growthequation} we need assumptions concerning the constitutive function $\widehat{\mathcal{G}}$ and the stored energy density $W$. The assumptions on $W$ and $\mathcal{G}$ will be formulated in Section~\ref{sec:hyperelastic} and Section~\ref{sec:existence}, respectively.

The paper is organized as follows: Section~\ref{sec:preliminiaries} recalls preliminary results that will be used thoughout the paper. In Section~\ref{sec:hyperelastic} the assumptions concerning the stored energy density $W$ are stated and uniform a~priori estimates are derived. The critical Lipschitz dependence of the stress and the nutrients on the growth tensor is proven in Sections~\ref{sec:stresslip} and~\ref{sec:nutrientslip}, respectively. Section~\ref{sec:existence} presents the assumptions on the growth dynamics $\mathcal{G}$ and the general existence and uniqueness theorem. In Section~\ref{sec:conclusions} some ideas concerning further developments are presented.

\bigskip

\textit{Notation:} Following customary notation in continuum mechanics (also used in the textbook~\cite{GorielyBook}), we use capital letters with a subscript $0$ to denote objects in the initial reference configuration, capital letters to denote the objects in the natural configuration and small letters to denote the objects in the current configuration, indicating the time $t$ as an argument. The independent spatial variable is denoted by $X$, $z$ and $x$ in the reference configuration, the natural configuration, and the deformed configuration, respectively. If $f$ is a function of one independent variable, we write $f'$ for the derivative. However, in order to underline the dependence of the domain, we also use the symbols $\partial_X$, $\partial_z$ and $\partial_x$ for total derivatives. 

We use standard notation for function spaces like Lebesgue spaces $L^p(I)$, Sobolev spaces $W^{k,p}(I)$, $H^k(I)=W^{k,2}(I)$, or spaces of differentiable functions $C^k(I)$ on an open interval $I$. If $I=[a,b]$ is a closed interval, we write $C^k([a,b])$ if the one-sided derivatives exist on the boundary of the interval. The space of all bounded and measurable functions endowed with the supremum norm is denoted by $\mathcal{L}^\infty([a,b])$. Partial derivatives are written as $\partial_p W = W_p$ and if a function $f(x,t)$ of two variables is considered as a function of one variable with the other variable fixed, then we use an index notation, that is, for $x$ fixed we write $f_{;x}(t)$. Here we write the semicolon in order to avoid confusion with partial derivatives. For example, $W_{pp;X}:(0,\infty)\to \R$ is defined by $W_{pp;X}(p) = \partial_{pp}W(X,p)$.

A significant part of the analysis is devoted to proving that dependent variables are Lipschitz continuous in their arguments. Suppose that $\chi =\chi(G)$ is a function, we frequently write the argument $G$ as an index, i.e., $\chi_G = \chi(G)$. Moreover, for $G_1$, $G_2$ we use the notation $\chi_i = \chi(G_i)$, $i=1,2$. For exampe, this is applied to $\chi\in \{\phi, y, S, W\}$ indicating the elastic deformation, the total deformation, the stress or the induced elastic energy. Moreover, in order to obtain global existence via Picard-Lindel\"of's theorem, it is crucial to obtain uniform control on the constants. This will be achieved by stating all estimates for a suitable ball in $\mathcal{L}^\infty([0,L_0])$. Fix $G_0\in \mathcal{L}^\infty([0,L_0])$, $R_1\in (0,\infty)$ with 
\begin{align}\label{boundsonG}
 \inf_{X\in [0,L_0]} G_0(X) - R_1 = \Gamma_0 >0\,, \quad \sup_{X\in [0,L_0]} G_0(X) + R_1 = \Gamma_1 <\infty\,.
\end{align}
Then, every $G\in B_1=B(G_0, R_1)\subset \mathcal{L}^\infty([0,L_0])$ satisfies $\inf G\geq \Gamma_0$ and $\sup G\leq \Gamma_1$. It is the goal of our analysis to verify that estimates involving $G\in B_1$ only depend on $\Gamma_0$ and $\Gamma_1$. We say that a constant on the right-hand side of an estimate depends only on global constants if the constant depends only on $\Gamma_0$, $\Gamma_1$, $R_1$ in~\eqref{boundsonG}, $L_0$, $\ell_0$, the elastic energy density $W$, e.g., through derivatives of $W$ and the function $\theta$ in (W2), and bounds on the coefficients in the reaction-diffusion equation in~\eqref{bounds:DandBeta}.

\section{Preliminaries}\label{sec:preliminiaries}

Before discussing the properties of the solutions of the ordinary differential equation which determines the growth dynamics and the induced variational problem which provides the elastic deformation, we recall some basic properties of Lipschitz functions that will be used in the verification of the assumptions of Picard-Lindel\"of's theorem. If $X$ and $Y$ are normed space, we denote the topological dual space of $X$ by $X'$ and we denote the space of all bounded and linear maps from $X$ into $Y$ by $\mathcal{L}(X,Y)$. For $f:U\subset X\to Y$ we define 
\begin{align*}
L_f = \sup_{x_1,x_2\in U,x_1\neq x_2} \frac{\| f(x_1) - f(x_2) \|_{Y}}{\| x_1 - x_2 \|_X}
\end{align*}
since the domain $U$ is always clear from the context. Any Lipschitz continuous function is bounded on bounded sets. For example, if $f:B(0,R)\subset X \to Y$ is Lipschitz continuous, for all $x\in B(0,R)$ the uniform estimate $\|f(x)\|_Y \leq L_f\cdot R + \| f(0) \|_Y$ follows.

\begin{lemma}\label{lem:PropertiesLip}
Let $X$, $Y$, $Z$ be normed spaces. 
\begin{itemize}
 \item [(a)] If $f:U\subset X\to Y$ and $g: V\subset Y\to Z$ are Lipschitz continuous functions with $f(U)\subset V$, $h=g \circ f:U\to Z$ is Lipschitz continuous with $L_h \leq L_f L_g$. 
 
\item [(b)] If $f:U\to Y$ and $f':U\to Y'$ are Lipschitz continuous and bounded, 
\begin{align*}
 h:U\to \R \,,\quad x\mapsto {_Y}\langle f(x), f'(x)\rangle_{Y'}
\end{align*}
is Lipschitz continuous with 
\begin{align*}
 L_h \leq \sup_{U}\| f' \|_{Y} L_f+ \sup_{U}\| f \|_Y L_{f'}\,. 
\end{align*}
An analogous formula holds for $f$, $g:U\to Y$ with $Y$ a Banach algebra and $h(x) = f(x)\cdot g(x)$.

\item [(c)] If $r$, $s\in (0,\infty)$ with $r<s$ and $f:U\subset X\to[r,s]$ is Lipschitz continuous and bounded,
\begin{align*}
h:U \to \R\,,\quad x\mapsto h(x) = \frac{1}{f(x)}
\end{align*}
is Lipschitz continuous with $L_h \leq L_f/r^2$.

\item [(d)] If $c_0$, $c_1\in (0,\infty)$ with $c_0<c_1$, $A_0\in \R$, $r$, $s\in \R$ with $r<s$,  $a\in L^\infty([r,s];[c_0,c_1])$,
\begin{align*}
A:[r,s]\to \R\,,\quad x\mapsto A(x) = \int_r^{x} a(\xi) \dv{\xi} + A_0
\end{align*}
is bi-Lipschitz, that is, $A^{-1}$ exists and $A$ and $A^{-1}$ are Lipschitz continuous with $L_A\leq c_1$ and $L_{A^{-1}}\leq 1/c_0$.
\end{itemize}
\end{lemma}

\begin{proof}
If the assumptions in (a) hold, then for all $x_1$ and $x_2\in U$, 
 \begin{align*}
  \| g(f(x_1)) - g(f(x_2)) \|_Z \leq L_g \| f(x_1) - f(x_2) \|_{Y}\leq L_f L_g \| x_1 - x_2 \|_X\,.
 \end{align*}
Similarly, in the situation (b) for all $x\in U$, 
\begin{align*}
& |{_Y}\langle f(x_1), f'(x_1)\rangle_{Y'} - {_Y}\langle f(x_2), f'(x_2)\rangle_{Y'}| \\ & \quad \leq |{_Y}\langle f(x_1)-f(x_2), f'(x_1)\rangle_{Y'}| + {_Y}\langle f(x_2),f'(x_1)- f'(x_2)\rangle_{Y'}| \\ & \quad \leq \sup_{U}\| f' \|_{Y} L_f+ \sup_{U}\| f \|_Y L_{f'}\,.
\end{align*}
To prove (c) for all $x_1$, $x_2\in U$,
\begin{align*}
 \Bbar{ \frac{1}{f(x_1)} - \frac{1}{f(x_2)} } = \frac{| f(x_1) - f(x_2)| }{|f(x_1) f(x_2)|}\leq \frac{L_f |x_1 - x_2|}{r^2}\,.
\end{align*}
Finally, in (d) $A$ is strictly monotonically increasing, $A^{-1}$ exists and for all $x_1$, $x_2\in [r,s]$
\begin{align*}
| A(x_1) - A(x_2) | = \Bbar{ \int_{x_1}^{x_2} a(\xi) \dv{\xi} } \leq c_1 |x_1 - x_2|\,.
\end{align*}
For $y_1$, $y_2 \in [A(r), A(s)]$  with $y_i = A(x_i)$, $i=1,2$ we find 
\begin{align*}
 |A^{-1}(y_1) - A^{-1}(y_2)| = |x_1 - x_2| \leq \Bbar{ \int_{x_1}^{x_2} \frac{a(\xi)}{c_0}\dv{\xi} }\leq \frac{1}{c_0} |A(x_2) - A(x_1)|\,.
\end{align*}
This argument completes the proof. 
\end{proof}

\begin{lemma}\label{lem:BoundInverseFunction}
Suppose that $\ell_0$, $L_0\in (0,\infty)$, $y_1$, $y_2\in W^{1,\infty}([0, L_0], [0,\ell_0])$ are strictly increasing, bijective, bi-Lipschitz and that there exist $\epsilon>0$, $a>0$ with $\| y_1 - y_2 \|_{C^0([0,L_0])}\leq \epsilon$ and $y_i'\geq a$ almost everywhere on $[0,L_0]$ for $i=1,2$. Then 
\begin{align*}
 \| y_1^{-1} - y_2^{-1} \|_{C^0([0,\ell_0])} \leq \frac{\epsilon}{a}\,.
\end{align*}
\end{lemma}

\begin{proof}
Fix $x\in [0, \ell_0]$ and define $X_i = y_i^{-1}(x)$, $i=1,2$. Then $y_1(X_1) = x = y_2(X_2)$ and 
\begin{align*}
 | y_1(X_1) - y_2(X_1) | = | y_2(X_2) - y_2(X_1) | =\Bbar{ \int_{X_1}^{X_2} y'_2(X)\dv{X} }\geq a |X_1 - X_2|
\end{align*}
and hence 
\begin{align*}
 | y_1^{-1}(x) - y_2^{-1}(x) | = |X_1 - X_2| \leq \frac{1}{a} |  y_1(X_1) - y_2(X_1) | \leq \frac{\epsilon}{a}
\end{align*}
and the assertion follows if we take the maximum in $x\in [0,\ell_0]$.
\end{proof}

Next we state a version of Picard-Lindel\"of's theorem which will be used in the existence proof.

\begin{theorem}[\mbox{\cite[Theorem~2.13]{SchechterNonlinearAnalysis2004}}]\label{PicardLindeloef}
Let $X$ be a Banach space, $x_0\in X$, $R_0$, $T_0>0$, $t_0\in \R$, and let
\begin{align*}
 B_0 = \{ x\in X\colon \| x - x_0 \|\leq R_0\}
\end{align*}
and 
\begin{align*}
 I_0 = \{ t\in \R\colon |t-t_0|\leq T_0\}\,.
\end{align*}
Assume that $g(\cdot,\cdot)$ is a continuous map of $I_0\times B_0$ into $X$ and that 
\begin{align*}
 \| g(t,x) - g(t,y) \|\leq K_0 \| x - y \|\quad x,\,y\in B_0,\,t\in I_0
\end{align*}
and that
\begin{align*}
 \| g(t,x) \| \leq M_0\quad x\in B_0,\,t\in I_0\,.
\end{align*}
Let $T_1$ be such that 
\begin{align*}
 T_1 \leq \min(T_0, R_0/M_0)\,,\quad K_0T_1<1\,.
\end{align*}
Then there exists a unique solution $x(\cdot)$ of 
\begin{align*}
 \frac{\mathrm{d}x(t)}{\mathrm{d}t} = g(t, x(t))\,,\quad |t-t_0|\leq T_1\,,\quad x(t_0) = x_0\,.
\end{align*}
\end{theorem}

\section{The hyperelastic variational problem}\label{sec:hyperelastic}
In this section, we recall the fundamental existence result for hyperelastic variational problems in one spatial dimension in \cite{Ball1981OneDimensional} and state the precise assumptions that are needed to ensure the existence of a minimizer on $g([0,L_0],t) = [g(0,t), g(L_0,t)]$, the natural configuration induced by a growth tensor $G\in B_1$ at time $t$. We drop the dependence on $t$, and consider $g=g(X)$ with $g'=\partial_X g=G$ as a function of the independent variable $X$ only. The variational problem can be formulated as follows: Suppose that $[g(0),g(L_0)]\subset \R$ is an interval with nonempty interior. It is the goal to minimize the variational integral 
 \begin{align}\label{def:IG}
  I_G[\phi] = \int_{g(0)}^{g(L_0)} W_G(z, \phi'(z))\dv{z} = \int_0^{L_0} W(X, \phi'(g(X)) G(X) \dv{X}
 \end{align}
in the class of admissible functions 
\begin{align}\label{def:AG}
\mathcal{A}_G = \nset{ \phi \in W^{1,1}(g(0),g(L_0)): I_G[\phi]<\infty,\,\phi(g(0))=0,\,\phi(g(L_0))=\ell_0 }
\end{align}
with $\ell_0>0$ given. Here $W_G$ denotes the free energy of the body in the natural configuration which, by assumption, is given by $W_G(z, \cdot) = W(g^{-1}(z),\cdot)$ for $z\in [g(0),g(L_0)]$. In the following we state the assumptions concerning $W$ that are needed in the existence and regularity results for the associated energy $W_G$. They include fundamental assumptions in continuum mechanics and do not impose upper bounds on the stored energy density. In Remark~\ref{rem:relationWG} we indicate the implications of the assumptions on $W$ on the function $W_G$.
\begin{itemize}[leftmargin=.4in]
 \item [(W1)] Regularity: $W\in C^0([0,L_0]\times (0,\infty);[0,\infty))$, $W(X,1)=0$ for all $X\in [0,L_0]$, and for all $X\in [0,L_0]$ fixed, $W(X, \cdot)\in C^1(0,\infty)$ with $\partial_p W\in C^0([0,L_0]\times (0,\infty))$.
 
 \item [(W2)] Structure and growth conditions: There exists a convex function $\theta:(0,\infty)\to \R$ with $\theta(p)\nearrow \infty$ as $p\searrow 0$ and $\theta(p)/p \nearrow \infty$ as $p\nearrow \infty$ with the following property: For all $X\in [0,L_0]$ fixed, $W(X,\cdot)$ is strictly convex with $W(X,p)\geq \theta(p)$ on $(0,\infty)$.

 \item [(W3)] For all $\lambda\in (0,\infty)$: $\int_0^{L_0} W(X, \lambda) \dv{X}<\infty$. 
 
 \item [(W4)] Additional regularity for $W$: $\partial_p W\in C^1( [0,L_0]\times (0,\infty))$ and $\partial_{pp} W>0$.
\end{itemize}

As a first implication of the foregoing assumptions we state an important consequence which follows from the uniform lower bound $\theta$ for $W(X, \cdot)$ and will serve as an a~priori estimate. 

\begin{lemma}\label{lem:algebraic}
Suppose that $W$ satisfies (W1)--(W2), that $S\in \R$, $E\subset [0,L_0]$, $E\neq \emptyset$, and that $\pi:E\to (0,\infty)$ is a function with the following property: For all $X\in E\colon \partial_p W(X, \pi(X))=S$. Then there exist constants $P_0$, $P_1\in (0,\infty)$ which depend only on $\theta$ and $S$ such that for all $X\in E\colon \pi(X) \in [P_0,P_1]$. Moreover, if $\Sigma_0$, $\Sigma_1\in \R$ with $\Sigma_0<\Sigma_1$, then $P_0$ and $P_1$ can be chosen uniformly for $S\in [\Sigma_0, \Sigma_1]$ and depend only on $\theta$, $\Sigma_0$ and $\Sigma_1$.
\end{lemma}

\begin{proof}
If $S=0$, then by (W1), (W2) $\pi\equiv 1$ on $E$ and it suffices to choose $P_0\leq 1$ and $P_1\geq 1$. Suppose next that $S<0$. Then by (W1), (W2) $\pi<1$ on $E$ and the upper bound follows with $P_1\geq 1$. To prove the lower bound, by (W1), (W2) for all $X\in E$ the derivative $\partial_p W(X, \cdot)$ is negative and strictly increasing on $[\pi(X),1)$ with $W(X,1)=0$ and by (W1), (W2), 
\begin{align*}
 \theta(\pi(X)) & \leq W(X, \pi(X)) = W(X, \pi(X)) - W(X, 1) = \int_1^{\pi(X)} \partial_p W(X, s) \dv{s} \\ & = \int_{\pi(X)}^{1} | \partial_p W(X, s)|\dv{s}\leq (1-\pi(X)) | \partial_p W(X, \pi(X))| \leq |S|\,.
\end{align*}
By (W2), $m=\min_{(0,\infty)}\theta \in(-\infty, 0]$, the set $M = \{ p\in (0,\infty): \theta(p)=m\}$ is bounded from below and we may define $p_0 = \min M>0$. By convexity, $\theta$ is strictly decreasing on $(0,p_0]$. In fact, if $a,b\in (0,p_0)$ with $a<b$ and $\theta(a)=\theta(b)>\theta(p_0)$, there exists a $\lambda\in (0,1)$ with
\begin{align*}
 \theta(b) = \theta(\lambda a + (1-\lambda)p_0) \leq \lambda \theta(a) + (1-\lambda) \theta(p_0) <\theta(b)\,,
\end{align*}
which yields a contradiction. Define $P_0(\theta, S) = (\theta|_{(0,p_0]})^{-1}(|S|)$. Since $\theta$ is strictly decreasing on $(0,p_0)$, $\theta>|S|$ on $(0, P_0(\theta, S))$ and hence $\pi(X) \geq P_0(\theta, S)$. Note that $P_0(\theta, S)$ is decreasing as a function of $S$ on $(-\infty, 0]$.

If $S>0$ then by (W1), (W2) $\pi>1$ on $E$ and it suffices to choose $P_0\leq 1$ to guarantee the lower bound. By (W1), (W2), for all $X\in E$ the derivative $\partial_p W(X, \cdot)$ is positive and strictly increasing on $(1, \pi(X)]$ with $W(X,1)=0$ and hence 
\begin{align*}
 \theta(\pi(X)) & \leq W(X, \pi(X)) = W(X, \pi(X)) - W(X, 1) = \int_1^{\pi(X)} \partial_p W(X, s) \dv{s} \\ & \leq (\pi(X)-1) \partial_p W(X, \pi(X)) \leq \pi(X) S\,.
\end{align*}
Since by (W2) $\theta(p)/p\nearrow \infty$ for $p\nearrow\infty$ and since $\theta(p)/p$ is continuous on $(0,\infty)$, we may define 
\begin{align*}
 P_1(\theta, S)  = \sup\{ p \in [1,\infty): \theta(p)/p=S\}<\infty\,.
\end{align*}
Since $\theta(\pi(X))/\pi(X) \leq S$, $\pi(X)\leq P_1(\theta, S)$ which establishes the upper bound. Note that $P_1(\theta, S)$ is increasing in $S$ on $[1,\infty)$. 

Finally, assume that $\Sigma_0$, $\Sigma_1\in \R$ with $\Sigma_0<\Sigma_1$. Define
\begin{align*}
 P_0(\theta, \Sigma_0, \Sigma_1) = P_0(\theta, \min\{\Sigma_0, 1\})\,,\quad P_1(\theta, \Sigma_0, \Sigma_1) = P_1(\theta, \max\{\Sigma_1, 1\})\,.
\end{align*}
This choice of $P_0$ and $P_1$ provides uniform bounds for $S\in [\Sigma_0, \Sigma_1]$ as asserted. 
\end{proof}

Now, we explore the consequences of (W4) in the subsequent lemma. The additional regularity is required in the case that the evolution equation for $G$ depends on the nutrients since the formula for the time-dependent diffusion coefficient introduces the term $\phi_G'\circ \phi_G^{-1}$ for which regularity can only be deduced from additional assumptions on $W$.

We first observe that assumption (W2) implies that for all $X\in [0,L_0]$ the map $\partial_p W(X, \cdot): (0,\infty)\to \R$, $p\mapsto \partial_p W(X, p)$ is strictly increasing and bijective. Thus, there exists a uniquely defined function $\pi_0:[0,L_0]\times \R \to (0,\infty)$ such that for all $X\in [0,L_0]$ and for all $S\in \R$ the equation $\partial_p W(X, \pi_0(X, S)) = S$ holds. We use the notation 
\begin{align*}
 \pi_0(X, S) = ( \partial_p W(X, \cdot))^{-1}(S) \,.
\end{align*}
The theorem on implicit functions implies that $\pi_0$ is differentiable in $X$. Before we prove this in detail, we note that for all $X$ the function $S\mapsto \pi_0(X,S)$ is continuous. To see this, suppose there were a sequence $(S_k)_{k\in \N}$ converging to $S\in \R$ for which $\pi_0(X, S_k)$ does not converge to $\pi_0(X,S)$. Since $S_k$ is bounded, Lemma~\ref{lem:algebraic} implies that $\pi_0(X, S_k)$ is uniformly bounded and we may assume without loss of generality that the sequence $\pi_0(X, S_k)$ converges to $\bar\pi \neq \pi_0(X, S)$. However, the continuity of $\partial_p W$ implies 
\begin{align*}
 \partial_p W(X, \bar \pi) = \lim_{k\to \infty} \partial_p W(X, \pi_0(X, S_k)) = \lim_{k\to \infty} S_k = S
\end{align*}
and, by the uniqueness of the solutions of this equation, $\pi_0(X, S) = \bar \pi$. This observation leads to a contradiction since this argument can be applied to any subsequence.

\begin{lemma}\label{lem:W5}
Suppose that $W$ satisfies (W1)--(W4). Then
\begin{align*}
\pi_0: (X,S)\mapsto \pi_0(X,S)= (\partial_p W(X,\cdot))^{-1}(S)
\end{align*}
is uniformly separately Lipschitz continuous on compact intervals. That is, for all $S_0$, $S_1\in \R$ with $S_0<S_1$ there exist constants $L_{\partial_p W^{-1},X}$ and $L_{\partial_p W^{-1},S}$ which depend only on the derivatives of $\partial_p W$ such that for all $X$, $\widetilde{X}\in [0,L_0]$, $S$, $\widetilde{S}\in [S_0, S_1]$
 \begin{align*}
 | \pi_0(X,S) - \pi_0(\widetilde{X},S)|& \leq L_{\partial_p W^{-1},X}(S_0, S_1)|X - \widetilde{X}|\,,\\ 
 | \pi_0(X,S) - \pi_0(X,\widetilde{S})| & \leq L_{\partial_p W^{-1},S}(S_0, S_1)|S - \widetilde{S}|\,.
 \end{align*}
\end{lemma}

\begin{proof}
We first consider $S\in \R$ fixed and prove that the map $\pi_0(\cdot, S):(0,L_0)\to (0,\infty)$ is differentiable. Suppose that $(X_0, p_0)\in (0,L_0)\times (0,\infty)$ is a solution of the equation $\partial_p W(X,p)=S$. Since $\partial_p W\in C^1((0,L_0)\times (0,\infty))$ with $\partial_{pp} W>0$ we may apply the theorem on implicit functions. Thus there exists an open neighborhood $U$ of $X_0$ in $(0,L_0)$, an open neighborhood $V$ of $p_0$ in $(0,\infty)$, and a differentiable function $p(\cdot, S):U\to V$ such that on $U\times V$ the equation $\partial_p W(X, p)=S$ holds if and only if $p=p(X, S)$. Since the solution of this equation with $S$ fixed is uniquely determined, necessarily $p(\cdot, S) = \pi_0(\cdot, S)$ on $U$ and since $X_0\in (0,L_0)$ is arbitrary, $p(\cdot, S)=\pi_0(\cdot, S)\in C^1(0,L_0)$. For $S\in \R$ fixed, Lemma~\ref{lem:algebraic} implies with $E=(0,L_0)$ and $\pi(\cdot) = \pi_0(\cdot, S)$ that the function $\pi_0(\cdot, S)$ is uniformly bounded with values in $[P_0,P_1]$. The derivative with respect to $X$ follows from implicit differentiation, 
\begin{align*}
 \partial_X \pi_0(X,S) = - \partial_X \partial_p W(X, \pi_0(X, S)) / \partial_{pp}W(X, \pi_0(X, S))
\end{align*}
and since $\pi_0(\cdot, S)$ is uniformly bounded, there exists a constant $c_0>0$ with $\partial_{pp} W\geq c_0$ on $[0,L_0]\times [P_0,P_1]$. Therefore, $\partial_X \pi_0(\cdot, S)$ is uniformly bounded on $(0,L_0)$ and these bounds imply that the Cauchy criterion for the existence of the limits $\lim_{X\to 0}\pi_0(X, S)$ and $\lim_{X\to L_0}\pi_0(X,S)$ is satisfied. Consequently, $\pi_0(\cdot, S)\in C^0([0,L_0])$. The local Lipschitz continuity of $\pi_0(\cdot, S)$ in the first argument follows from the mean value theorem and the assumption $\partial_p W\in C^1([0,L_0]\times (0,\infty))$ since the bounds in Lemma~\ref{lem:algebraic} are uniform for $S$ in compact intervals.

In the discussion preceding this lemma we already proved that $\pi_0(X, \cdot)$ is continuous as a function in the second argument. Thus, it remains to prove differentiability in $S$ with uniform bounds. Fix $X\in [0,L_0]$, $S\in \R$ and $\Delta S\neq 0$ with $\partial_p W(X, \pi_0(X,S)) = S$, $\partial_p W(X, \pi_0(X, S+\Delta S)) = S+ \Delta S$. Then 
\begin{align*}
\Delta S & = \partial_p W(X, \pi_0(X,S+\Delta S)) - \partial_p W(X, \pi_0(X, S))  \\ & = \int_0^1 \partial_{pp} W(X, t\pi_0(X, S+\Delta S) + (1-t) \pi_0(X,S)) (\pi_0(X,S+\Delta S) - \pi_0(X,S)) \dv{t} 
\end{align*} 
and hence, in view of $\partial_{pp} W\geq c_0>0$,
\begin{align*}
 \frac{\pi_0(X, S+\Delta S) - \pi_0(X,S))}{\Delta S} = \Bsqb{ \int_0^1 \partial_{pp} W(X, t\pi_0(X, S+\Delta S) + (1-t) \pi_0(X,S)) \dv{t} }^{-1}\,.
\end{align*}
By the continuity of $\pi_0$ in $S$ we may pass to the limit $\Delta S\to 0$ and obtain 
\begin{align*}
 \partial_S \pi_0(X, S) = \Bsqb{\partial_{pp} W(X,\pi_0(X,S)) }^{-1}\,.
\end{align*}
The explicit formula implies that $\partial_S \pi_0$ is uniformly bounded for $X\in [0,L_0]$ and $S\in \R$ on compact intervals.
\end{proof}

\begin{remark}\label{rem:relationWG}
The subsequent existence theorem in Theorem~\ref{thm:BallExistence1d} is applied to the functional $I_G$ defined via the energy density $W_G(z,\cdot) = W(g^{-1}(z),\cdot)$ where $g$ is bi-Lipschitz. In fact, if (W1)--(W3) hold for $W$, then the induced function $W_G$ satisfies (W1)--(W3) as well with $[0,L_0]$ replaced by $[g(0), g(L_0)]$. Moreover, since the change of variables concerns only the independent variables, (W2) is satisfied with a fixed function $\theta$ independent of $g$. Assumption (W3) ensures that the class of admissible function $\mathcal{A}_G$ is not empty. For $G$ the identity map, we define $\mathcal{A}=\nset{ \phi \in W^{1,1}(0,L_0): I[\phi]<\infty,\,\phi(0)=0,\,\phi(L_0)=\ell_0 }$.
\end{remark}

\begin{theorem}[\cite{Ball1981OneDimensional}]\label{thm:BallExistence1d}
 Suppose that $W$ satisfies (W1)--(W3). Then there exists a function $\phi\in \mathcal{A}$ which minimizes $I$ in $\mathcal{A}$. Moreover, $\phi\in C^1([0,L_0])$ with $\min_{[0,L_0]}\phi' > 0$ and the Euler-Lagrange equation 
 \begin{align*}
  \frac{\mathrm{d}}{\mathrm{d}X} (\partial_p W(X,\phi'(X)))=0
 \end{align*}
holds on $[0,L_0]$. In particular, the stress $\partial_p W(X, \phi'(X))$ is constant on $[0,L_0]$.
\end{theorem}

\begin{proof}
The assertion follows from Theorems~1 and~2 in \cite{Ball1981OneDimensional}. For completeness, we verify the assumptions (H) in \cite{Ball1981OneDimensional}: (H1), $W$ being a Carath\'eodory function, is weaker than our assumption (W1); the growth conditions (H2), $W(X,p)\to \infty$ for $p\to 0^+$ and almost all $X\in [0,L_0]$, and (H3), $W(X,p)\geq \psi(p)$ for all $p$ and almost all $X$ with $\psi$ convex and $\psi(p)/p\to \infty$ for $p\to\infty$, follow from (W2); (H4) is (W3); (H6), $W(X,\cdot)$ is $C^1$ for almost all $X\in [0,L_0]$ and $W_p$ is a Carath\'eodory function, is contained in (W1); finally the additional assumptions in Theorem~2 needed to conclude that $\phi\in C^1([0,L_0])$ concerning continuity and strict convexity are included in (W1) and (W2). Since $\phi\in C^1([0,L_0])$ the Euler-Lagrange equation holds pointwise. In fact, the proof in~\cite{Ball1981OneDimensional}  shows that there exists a constant $C$ with $\partial_p W(X,\phi'(X))=C$ a.e.\ on $[0,L_0]$. Since $\phi\in C^1([0,L_0])$ and $\partial_p W\in C^0((0,L_0)\times (0,\infty))$, the equation holds in $[0,L_0]$. 
\end{proof}

Theorem~\ref{thm:BallExistence1d} corresponds to the case $G\equiv 1$ in the following Corollary.

\begin{corollary}\label{cor:ExistenceIG}
Suppose that $G\in B_1=B(G_0,R_1)$ with~\eqref{boundsonG} and that $W$ satisfies (W1)--W(3). Then there exists a minimizer $\phi_G$ of $I_G$ in~\eqref{def:IG} in the class $\mathcal{A}_G$ in~\eqref{def:AG}. Moreover, $\phi_G\in C^1([g(0),g(L_0)])$ with $\min_{[g(0),g(L_0)]}\phi_G' > 0$ and the Euler-Lagrange equation 
 \begin{align*}
  \frac{\mathrm{d}}{\mathrm{d}z} (\partial_p W_G(z, \phi_G'(z)))=0
 \end{align*}
holds on $[g(0),g(L_0)]$. In particular, the stress $\partial_p W_G(z, \phi_G'(z))$ is constant on $[g(0),g(L_0)]$.
\end{corollary}

\begin{proof}
This follows from the observations concerning $W_G$ in Remark~\ref{rem:relationWG}.
\end{proof}

\begin{corollary}
Suppose that the assumptions in Theorem~\ref{thm:BallExistence1d} or Corollary~\ref{cor:ExistenceIG} hold. Then the corresponding variational problems have unique minimizers. 
\end{corollary}

\begin{proof}
This follows from the strict convexity of $W(X, \cdot)$ and $W_G(z, \cdot)$, respectively.
\end{proof}

\begin{lemma}
\label{lem:BoundStress}
Suppose that $W$ satisfies (W1)--(W3) and that $B_1=B(G_0,R_1)$ with~\eqref{boundsonG}. Then there exist constants $P_0$, $P_1$, $\Sigma_0$, $\Sigma_1\in \R$, which only depend on~\eqref{boundsonG} and the function $\theta$ in (W2) and $W$ with $0<P_0<P_1$ and $\Sigma_0<\Sigma_1$ and the following property: If $G\in B_1$ with~\eqref{boundsonG} and if $\phi_G\in C^1([g(0),g(L_0)])$ denote the minimizer obtained in Corollary~\ref{cor:ExistenceIG}, then $S_G\in [\Sigma_0,\Sigma_1]$ and $\phi_G'\in[P_0,P_1]$ on $[g(0),g(L_0)]$.
\end{lemma}

\begin{proof}
Fix $G\in B_1$ with the properties in the assertion of the lemma. If $L_G = g(L_0) - g(0) = \ell_0$, then $\phi_G = \id|_{[g(0),g(L_0)]}$ is the unique minimizer of $I_G$ in $\mathcal{A}_G$ with $\partial_p W(\cdot, \phi_G')=0$ and it suffices to choose $\Sigma_0\leq 0 \leq \Sigma_1$. 

Suppose next that $L_G > \ell_0$. By construction, $L_G\leq \Gamma_1 L_0$ and by the mean value theorem there exists a $\xi\in (g(0), g(L_0))$ with 
\begin{align*}
 \phi_G'(\xi) = \frac{\phi_G(g(L_0)) - \phi_G(g(0))}{g(L_0) - g(0)} = \frac{\ell_0 - 0}{L_G} <1\,,
\end{align*}
and, since $\partial_p W_G(\xi, \cdot)$ is increasing with a unique zero in $p=1$, 
\begin{align*}
 0 & > S_G = \partial_p W_G(\xi, \phi_G'(\xi)) = \partial_p W_G(\xi, \frac{\ell_0}{L_G}) \geq \partial_p W_G(\xi, \frac{\ell_0}{\Gamma_1 L_0}) \\ & = \partial_p W(g^{-1}(\xi), \frac{\ell_0}{\Gamma_1 L_0}) \geq \inf_{X\in [0,L_0]}\partial_p W(X, \frac{\ell_0}{\Gamma_1 L_0})=\Sigma_0>-\infty\,.
\end{align*}
If $L_G < \ell_0$ and $L_G \geq \Gamma_0 L_0$, then there exists, by the mean value theorem, a $\xi \in (g(0), g(L_0))$ with 
\begin{align*}
 \frac{\phi_G(g(L_0)) - \phi_G(g(0))}{g(L_0) - g(0)} = \frac{\ell_0 - 0}{L_G} = \phi_G'(\xi)>1\,,
\end{align*}
and, since $\partial_p W_G(\xi, \cdot)$ is increasing, 
\begin{align*}
 0 & < S_G = \partial_p W_G(\xi, \phi_G'(\xi)) = \partial_p W_G(\xi, \frac{\ell_0}{L_G}) \leq \partial_p W_G(\xi, \frac{\ell_0}{\Gamma_0 L_0}) \\ & = \partial_p W(g^{-1}(\xi), \frac{\ell_0}{\Gamma_0 L_0}) \leq \sup_{X\in [0,L_0]}\partial_p W(X, \frac{\ell_0}{\Gamma_1 L_0})=\Sigma_1< \infty\,.
\end{align*}
Hence, for all $G\in B_1$, $S_G\in [\Sigma_0, \Sigma_1]$ and $\Sigma_0$, $\Sigma_1$ depend only on $\ell_0$, $L_0$, $\Gamma_0$, $\Gamma_1$ and $W$. The assertion follows from Lemma~\ref{lem:algebraic} with $E=[0,L_0]$ and $\pi(X) = \phi_G'(g(X))$ since $\phi_G$ satisfies the Euler-Lagrange equation $\partial_p W_G(z, \phi_G'(z))=S_G$ and since $g:[0,L_0]\to [g(0), g(L_0)]$ is bijective. 
\end{proof}

\begin{example}
Consider a body for which the elastic response is determined by a free energy density $W$ with $X$ dependent modulus of elasticity, that is, $W(X,p) = \kappa(X) W(p)$ where $W\in C^2(0,\infty)$ satisfies (W1) and (W2). If the stress in the system is given, $\partial_p W(X,p)=S$, then we can solve for $p$ due to the strict monotonicity of the bijective map $\partial_p W:(0,\infty)\to \R$ and find 
\begin{align*}
 \kappa(X) \partial_p W(p) = S \quad \Leftrightarrow \quad p = (\partial_p W)^{-1}\Barg{\frac{S}{\kappa(X)}}
\end{align*}
If $\kappa$ is Lipschitz continuous and bounded, $\kappa\in W^{1,\infty}(0,L_0;[\kappa_0,\kappa_1])$ with $\kappa_0>0$, and if $S$ is contained in a compact interval $S\in [S_0, S_1]$, then the assertions in Lemma~\ref{lem:W5} are satisfied since $(\partial_p W)^{-1}$ is continuously differentiable with bounded derivative on compact sets. A typical example for $W:(0,\infty)\to [0,\infty)$ is the function $p\mapsto W(p)= (p-1/p)^2$.
\end{example}

\section{Lipschitz dependence of the stress on the growth tensor}\label{sec:stresslip}

In Example~\ref{ex:stresscoupling} a model system was introduced in which the ordinary differential equation for $G(X,t)$ is nonlocal since it depends on the stress $S(t)$ which is determined from the elastic response of the material on $[g(0,t),g(L_0,t)]$. In this section we prove that the map $G\mapsto S(G)$ is Lipschitz continuous. To verify this fact, we consider this map as a composition of three maps: The maps which associate to a growth tensor $G$ the growth deformation $g$, the elastic deformation $\phi$, and finally the constant stress $S$ in the body . Based on this result, existence for the system in Example~\ref{ex:stresscoupling} can be proven by Picard-Lindel\"of's theorem, see Theorem~\ref{thm:ExistenceFullCoupling} and the discussion in Section~\ref{sec:conclusions}.

\begin{lemma}\label{lem:LipS}
Suppose that $W$ satisfies (W1)--(W4) and that $B_1=B(G_0,R_1)$ with~\eqref{boundsonG}. Then the map $S:B_1\to \R$, $S\mapsto S(G)$ is of class $C^1$ and there exist constants $M_{S,G}$ and $L_{S,G}$ which depend only on $G_0$ and $R_1$ such that
 \begin{align*}
 \| S \|_{C^0(B_1)} & \leq M_{S,G}\,,\\
   \sup_{G\in B_1} \Bnorm{ \frac{\partial S}{\partial G}(G) }_{\mathcal{L}(\mathcal{L}^\infty([0,L_0]);\R)} & \leq L_{S,G}\,.
 \end{align*}
In particular, $S$ is globally Lipschitz continuous on $B_1$ with Lipschitz constant $L_{S,G}$.
\end{lemma}

\begin{proof} 
By assumption the estimates in~\eqref{boundsonG} hold for all $G\in B_1$ and by Lemma~\ref{lem:BoundStress} there exist $\Sigma_0$, $\Sigma_1\in \R$, $\Sigma_0\leq \Sigma_1$ with $S(G)\in [\Sigma_0, \Sigma_1]$. All constants depend only on $G_0$, $R_1$, $W$, and $\theta$ since the constants $\Sigma_0$, $\Sigma_1$ depend, in view of Lemma~\ref{lem:BoundStress}, only on $G_0$, $R_1$, $W$. For $G\in B_1$ fixed we denote by $\phi_G\in C^1([g(0), g(L_0)])$ the unique minimizer in 
\begin{align*}
 \mathcal{A}_G = \nset{ \phi\in W^{1,1}(g(0),g(L_0)), I_G[\phi]<\infty, \phi(g(0))=0, \phi(g(L_0)) = \ell_0 }\subset W^{1,1}(g(0), g(L_0))
\end{align*}
of the variational integral 
\begin{align*}
 I_G[\phi] = \int_{g(0)}^{g(L_0)} W_G(z, \phi'(z)) \dv{z}\,,
\end{align*}
the existence and regularity of which is guaranteed by Corollary~\ref{cor:ExistenceIG}. Moreover, there exists an $S(G)\in \R$ such that the Euler-Lagrange equation 
\begin{align*}
 W_p(g^{-1}(z), \phi_G'(z)) = S(G) \quad \text{ on } [g(0),g(L_0)]
\end{align*}
holds. Since $g$ is bijective, the substitution $X=g^{-1}(z)$ or $g(X)=z$ implies that for all $X\in [0,L_0]$
\begin{align*}
\partial_p W(X, \phi_G'(g(X))) = S(G)\quad \text{ or }\quad (\phi_G'\circ g)(X) = \pi_0(X, S(G))
\end{align*}
and since $g'$ exists for almost all $X$, for almost all $X\in [0,L_0]$
\begin{align*}
\phi_G'(g(X)) g'(X) =  (\partial_p W(X, \cdot))^{-1}(S(G)) g'(X) = \pi_0(X,S(G))G(X)\,.
\end{align*}
By the properties of $\pi_0(X,S) = (\partial_p W(X, \cdot))^{-1}(S)$ in Lemma~\ref{lem:W5}, $\phi_G' \circ g$ is continuous as a function of $X$ on $[0,L_0]$, hence measurable, and an integration in $X$ shows that 
\begin{align*}
 \ell_0 = \phi_G(g(L_0)) - \phi_G(g(0)) = \int_0^{L_0} \frac{\mathrm{d}}{\mathrm{d}X} [ \phi_G\circ g](X) ] \dv{X} = \int_0^{L_0}  \pi_0(X,S(G)) G(X) \dv{X}\,.
\end{align*}
Define $\Phi :(\Sigma_0-1, \Sigma_1+1)\times B_1\to \R$ by 
\begin{align*}
 (S, G) \mapsto \Phi(S, G) = \int_0^{L_0} \pi_0(X,S) G(X) \dv{X} - \ell_0
\end{align*}
and note that each pair $(S(G), G)$ satisfies the equation $\Phi(S(G), G)=0$. Moreover, since $\partial_p W(X,\cdot):(0,\infty)\to \R$ is strictly increasing and onto, also $\partial_p W(X,\cdot)^{-1}:\R\to (0,\infty)$ is strictly increasing and onto, thus, the integral is strictly increasing as a function of $S$ as well and there exists, for every $G\in B_1$, exactly one $S\in [\Sigma_0, \Sigma_1]$ such that $\Phi(S,G)=0$. We use the theorem on implicit functions to prove that the map $G\mapsto S(G)$ is locally of class $C^1$ and that $\partial S/\partial G$ is uniformly bounded in $B_1$. In order to prove that $\Phi\in C^1$, it suffices by~\cite[Prop.~1.2]{AmbrosettiMalchiodi2007} to show that the partial derivatives exist and are continuous. Fix $(S, G)\in  (\Sigma_0-1,\Sigma_1+1)\times B_1$. Since $\Phi$ is linear in $G$, for all $\Delta G\in \mathcal{L}^\infty([0,L_0])$
\begin{align*}
 \frac{\partial \Phi}{\partial G}(S,G)[\Delta G] = \int_0^{L_0} \pi_0(X,S) \Delta G(X) \dv{X}\,.
\end{align*}
This map is linear and bounded since 
\begin{align*}
\sup_{0\neq \Delta G\in \mathcal{L}^\infty([0,L_0])} \Bnorm{ \frac{\partial \Phi}{\partial G}(S,G)[\Delta G] }_{\mathcal{L}(\mathcal{L}^\infty([0,L_0]);\R)} \cdot \| \Delta G \|_{\mathcal{L}^\infty([0,L_0])}^{-1} \leq \int_0^{L_0} \pi_0(X,S) \dv{X}
\end{align*}
and by Lemma~\ref{lem:W5} the integrand is continuous as a function of $X$ and uniformly bounded for $X\in [0,L_0]$ and $S\in [\Sigma_0-1, \Sigma_1+1]$. Moreover, this derivative is continuous on its domain $(\Sigma_0-1,\Sigma_1+1)\times B_1$ since the integrand is Lipschitz continuous in $S$, uniformly bounded in $S$ and $X$ on compact sets and independent of $G$. For $\Delta S\in \R$ we find
\begin{align*}
 \frac{\partial\Phi}{\partial S}(S, G) [\Delta S] & = \frac{\mathrm{d}}{\mathrm{d}\epsilon}\Bigr|_{\epsilon=0}\Phi(S+\epsilon \Delta S, G) = \frac{\mathrm{d}}{\mathrm{d}\epsilon}\Bigr|_{\epsilon=0} \int_0^{L_0} \pi_0(X,S+\epsilon \Delta S) G(X) \dv{X}  \\ & = \int_{0}^{L_0}\frac{\mathrm{d}}{\mathrm{d}\epsilon}\Bigr|_{\epsilon=0} \pi_0(X,S + \epsilon \Delta S) G(X) \dv{X} \\ & =\int_{0}^{L_0}  \frac{\partial}{\partial S} (\partial_p W(X, \cdot))^{-1}  (S)\Delta S G(X) \dv{X} \\ & = \Bsqb{ \int_{0}^{L_0} \frac{1}{\partial_{pp} W(X,\pi_0(X,S))} G(X)\dv{X}} \Delta S
\end{align*}
and we may identify the derivative with real number.
Since $S$ is contained in the compact interval $[\Sigma_0-1,\Sigma_1+1]$ and, since by Lemma~\ref{lem:W5} $\pi_0(\cdot,\cdot)$ is separately Lipschitz continuous on compact intervals, $\{\pi_0(X,S), X\in [0,L_0], S\in [\Sigma_0-1,\Sigma_1+1]\}$ is contained in a compact set and $\partial_{pp} W(X,\pi_0(\cdot, \cdot))$ is uniformly bounded from below and from above. Finally, $G\in B_1$ is bounded from below by $\Gamma_0$ and therefore $\partial\Phi/\partial S$ is uniformly bounded from below on $(S_0-1, S_1+1)\times B_1$ by a positive constant which depends only on the global constants. The continuity of this derivative  on its domain $(\Sigma_0-1,\Sigma_1-1)\times B_1$ follows from the continuity of $\pi_0$ in Lemma~\ref{lem:W5} and the theorem on dominated convergence. 

Thus, we may apply the theorem on implicit functions. Suppose that $\Phi(S,G)=0$. Then there exists an open neighborhood $U$ of $S$ in $\R$ and an open neighborhood $V$ of $G$ in $\mathcal{L}^\infty([0,L_0])$ and a $C^1$ function $\widetilde{S}:V\to U$, $G\mapsto \widetilde{S}(G)$ such that $\Phi(S,G) = 0$ in $U\times V$ if and only if $S = \widetilde{S}(G)$. Since the stress is uniquely determined from $G$, $\widetilde{S}(G)=S(G)$ on $V$ and the regularity of $S$ has been established. 

The formula for the differentiation of implicit functions shows that for all $\Delta G\in \mathcal{L}^\infty([0,L_0])$
\begin{align*}
 \frac{\partial S}{\partial G}(G)[\Delta G] = - \Barg{ \frac{\partial\Phi}{\partial S}(S, G) }^{-1}\Bsqb{\frac{\partial \Phi}{\partial G}(S, G)[\Delta G]}
\end{align*}
and thus $\partial S/ \partial G$ is uniformly bounded in $C^0( (\Sigma_0-1, \Sigma_1+1)\times B_1, \mathcal{L}(\mathcal{L}^\infty([0,L_0]),\R))$. Finally the mean value theorem in Banach spaces implies for $G_1$, $G_2\in B_1$, in view of the convexity of $B_0$, that 
\begin{align*}
 | S(G_1) - S(G_2) | \leq C(\Gamma_0,\Gamma_1) \| G_1 - G_2 \|_{\mathcal{L}^\infty([0,L_0])}\,.
\end{align*}
Therefore, $S$ is globally Lipschitz continuous on $B_1$. 
\end{proof}

\section{Lipschitz dependence of the nutrients on the growth tensor}\label{sec:nutrientslip}

The inclusion of a dependence on nutrients in the evolution equation for the growth tensor $G$ leads to additional difficulties since the domain of $n_G$ is the deformed configuration, normalized to be the interval $[0,\ell_0]$. Therefore, Lipschitz continuity of the map $y_G = \phi_G \circ g$ is needed in order to control the change of variables in $N_G = n_G \circ y_G$. Moreover, the assumptions in our model lead to the factor $\phi_G' \circ \phi_G^{-1}$ in the expression for the diffusion coefficient $D_G$ and the reaction rate $\beta_G$. In the following, we illustrate how to obtain the Lipschitz continuity on $B_1 = B(G_0, R_1)$ for this specific model. The analysis of similar models is expected to be analogous. We begin by investigating the dependence of the elastic deformation $\phi_G$ on $G$. Since $\phi_G$ is defined on $[g(0), g(L_0)]$, it is necessary to consider $\phi_G \circ g:[0,L_0]\to \R$. 

\begin{lemma}
\label{lem:lipconty}
Suppose that the assumptions (W1)--(W4) hold and that $B_1 = B(G_0, R_1)$ with~\eqref{boundsonG}. Then there exist constants $L_{y,G}$ and $L_{\phi'\circ g,G}$ which depend only on global constants such that for all $G_1$, $G_2\in B_1$ the following assertions hold: 
\begin{itemize}
 \item [(i)] The map $y:B_1\to C^0([0,L_0])$, $G\mapsto y_G = \phi_G \circ g$ is Lipschitz continuous, i.e., 
\begin{align*}
  \| y_1 - y_2 \|_{C^0([0,L_0])} =   \| \phi_1 \circ g_1 - \phi_2\circ g_2 \|_{C^0([0,L_0])} \leq L_{y,G} \| G_1 - G_2 \|_{\mathcal{L}^\infty([0,L_0])}\,.
\end{align*}

\item [(ii)] The map $\phi'\circ g:B_1\to C^0([0,L_0])$, $G\mapsto \phi_G' \circ g$ is Lipschitz continuous, i.e., 
\begin{align*}
 \| \phi_1' \circ g_1 - \phi_2'\circ g_2 \|_{C^0([0,L_0])} \leq L_{\phi'\circ g,G} \| G_1 - G_2 \|_{\mathcal{L}^\infty([0,L_0])}\,.
\end{align*}
\end{itemize}
\end{lemma}

\begin{proof}
By Lemma~\ref{lem:BoundStress} there exist $P_0$, $P_1$, $\Sigma_0$, $\Sigma_1\in \R$ which depend only on global constants such that for all $G\in B_1$ the bounds $\phi_G'\in[P_0,P_1]$ and $S_G\in [\Sigma_0, \Sigma_1]$ are true. By Corollary~\ref{cor:ExistenceIG} the Euler-Lagrange equation implies that the stress is constant and that there exists a constant $S_G\in [\Sigma_0,\Sigma_1]$ with 
\begin{align*}
\partial_p W(g^{-1}(\cdot), \phi_G'(\cdot)) = S_G \quad \text{ on }[g(0),g(L_0)]
\end{align*}
and, since $g:[0,L_0]\to [g(0),g(L_0)]$ is bijective, on $[0,L_0]$,
\begin{align*}
 (\phi_G'\circ g)(\cdot) = \pi_0(\cdot, S_G)\,.
\end{align*}
Lemma~\ref{lem:W5} ensures that $\pi_0$ is Lipschitz continuous as a function of $S$ on $[\Sigma_0, \Sigma_1]$ and that the Lipschitz constant is independent of $X\in [0,L_0]$. Since the map $G\mapsto S_G$ is Lipschitz continuous by Lemma~\ref{lem:LipS}, we obtain that for all $G_1$, $G_2\in B_1$ 
\begin{align*}
\| \phi_1' \circ g_1 - \phi_2'\circ g_2 \|_{C^0([0,L_0])} & = \sup_{X\in [0,L_0]} |\phi_1' \circ g_1 - \phi_2'\circ g_2 |(X) = \sup_{X\in [0,L_0]} | \pi_0(\cdot,S_1) - \pi_0(\cdot,S_2)|(X)  \\  & \leq L_{W_p^{-1},S}(\Sigma_0,\Sigma_1) |S_1 - S_2| \\ & \leq L_{W_p^{-1},S}(\Sigma_0,\Sigma_1) L_{S,G} \| G_1 - G_2 \|_{\mathcal{L}^\infty([0,L_0])}\,,
\end{align*}
that is, the second assertion. To prove the first assertion, recall that $(\phi_i \circ g_i)(0)=0$, and that $\phi_i \circ g_i\in W^{1,\infty}(0,L_0)$, $i=1,2$. By the fundamental theorem of calculus and the chain rule, see~\cite[Lemma 7.5]{GilbargTrudinger1998}, for all $X\in [0,\ell_0]$ 
\begin{align*}
|y_1(X)-y_2(X)|&= | (\phi_1\circ g_1)(X) - (\phi_2\circ g_2)(X) | \leq \int_0^{X} | (\phi_1\circ g_1)'(U) - (\phi_2\circ g_2)'(U) |\dv{U} \\ & \leq L_0 \| (\phi_1'\circ g_1)g_1' - (\phi_2'\circ g_2)g_2'\|_{L^\infty(0,L_0)} \\ & \leq L_0 \| [(\phi_1'\circ g_1) - (\phi_2'\circ g_2)]g_1'\|_{L^\infty(0,L_0)} + L_0 \| g_1' - g_2' \|_{L^\infty(0,L_0)} \cdot \| \phi_2'\circ g_2 \|_{L^\infty(0,L_0)}
\end{align*}
and hence in view of (ii) and the global bounds $\phi_i'\in[P_0,P_1]$, $i=1,2$, 
\begin{align*}
\| y_1 - y_2\|_{C^0([0,L_0])}  = \| \phi_1 \circ g_1 - \phi_2\circ g_2 \|_{C^0([0,L_0])} 
& \leq  L_0(\Gamma_1 L_{\phi'\circ g,G} + P_1)  \| G_1 - G_2 \|_{L^\infty(0,L_0)}\,.
\end{align*}
Thus, the proof of the first assertion is concluded since for $G\in \mathcal{L}^\infty([0,L_0])$ and $g(X) = \int_0^X g(U)\dv{U}$ the identity $g'=G$ holds almost everywhere.
\end{proof}

Before we prove Lipschitz continuity of the diffusion coefficient as a function of $G$ defined on $B_1$ we derive uniform Lipschitz estimates for the functions $X\mapsto y_G(X)$, $x\mapsto (\phi_G'\circ \phi_G^{-1})(x)$ and, as a consequence, as stated in the Lemma~\ref{lem:PropertiesReactionDiffusion}, for $x\mapsto D_G(x)$.

\begin{lemma}\label{lem:LipInX}
 Suppose that (W1)--(W4) hold and that $B_1 = B(G_0,R_1)$ with~\eqref{boundsonG}. Then there exist constants $L_{y,X}$ and $L_{\phi'\circ \phi^{-1},x}$ such that the following assertions hold:
 \begin{itemize}
  \item [(i)] for all $G\in B_1$ and for all $X_1$, $X_2\in [0,L_0]$
\begin{align*}
 | y_G(X_1) - y_G(X_2) | \leq L_{y,X} | X_1 - X_2|\,;
\end{align*}

\item [(ii)] for all $G\in B_1$ and for all $x_1$, $x_2\in [0,\ell_0]$
\begin{align*}
 | (\phi_G' \circ \phi_G^{-1})(x_1) - (\phi_G' \circ \phi_G^{-1})(x_2) | \leq L_{\phi'\circ\phi^{-1},x} |x_1 - x_2|\,;
\end{align*}

\item [(iii)]  for all $G\in B_1$ and for all $X_1$, $X_2\in [0,L_0]$
\begin{align*}
 | (\phi_G' \circ g)(X_1) - (\phi_G' \circ g)(X_2) | \leq L_{\phi'\circ g,X} |X_1 - X_2|\,.
\end{align*}
\end{itemize}
\end{lemma}

\begin{proof}
 By Lemma~\ref{lem:BoundStress} there exist $P_0$, $P_1$, $\Sigma_0$, $\Sigma_1\in \R$ which depend only on the global constants such that for all $G\in B_1$ the bounds $\phi_G'\in [P_0, P_1]$ and $S_G\in [\Sigma_0, \Sigma_1]$ hold. Since for almost all $X\in [0,L_0]$, $g'(X)=G(X)\in [\Gamma_0, \Gamma_1]$, both $g$ and $\phi_G$ are Lipschitz continuous and hence $y_G$ is Lipschitz continuous. The Lipschitz constant depends only on global constants and this fact proves (i). To prove (ii), recall that for all $G\in B_1$ the Euler-Lagrange equation $\partial_p W_G(z, \phi_G'(z))=S_G$ holds and since the map $g:[0,L_0]\to [g(0),g(L_0)]$ is bijective, we find for all $X\in [0,L_0]$ that $\partial_p W(X, (\phi_G'\circ g)(X)) = S_G$ and this equation is equivalent to $(\phi_G'\circ g)(X) = \pi_0(X, S_G)$. By Lemma~\ref{lem:W5}, $\phi_G'\circ g$ is Lipschitz continuous in $X\in [0,L_0]$ and the Lipschitz constant is uniformly bounded in $S\in [\Sigma_0, \Sigma_1]$, proving (iii). Therefore, 
 \begin{align*}
  \phi_G' \circ \phi_G^{-1} = (\phi_G' \circ g) \circ (g^{-1}\circ \phi_G^{-1}) = (\phi_G' \circ g) \circ y_G^{-1}
 \end{align*}
is Lipschitz continuous by Lemma~\ref{lem:PropertiesLip} since 
\begin{align*}
 y_G(X) = y_G(X) - y_G(0) = \int_0^X y_G'(U)\dv{U} = \int_0^X (\phi_G' \circ g)(U) g'(U) \dv{U} 
\end{align*}
with $\phi_G'\circ g \in [P_0,P_1]$ and $g'\in [\Gamma_0, \Gamma_1]$ and $P_0$, $\Gamma_0>0$.
\end{proof}

Recall, that throughout the paper we assume that the assumptions~\eqref{bounds:DandBeta} concerning the reaction-diffusion equation hold.

\begin{lemma}
\label{lem:PropertiesReactionDiffusion}
Suppose that (W1)--(W4) hold and that $B_1 = B(G_0, R_1)$ with \eqref{boundsonG}. Then the mappings 
\begin{align*}
 D:B_1 \to C^0([0,\ell_0])\,,\quad & G\mapsto D_G = (\phi_G'\circ \phi_G^{-1})\cdot  (D_0\circ y_G^{-1}) \,, \\
\beta:B_1 \to \mathcal{L}^\infty([0,\ell_0])\,,\quad & G\mapsto \beta_G = [\phi_G'\circ \phi_G^{-1}]^{-1}\cdot (\beta_0\circ y_G^{-1})
\end{align*}
are bounded and Lipschitz continuous, that is, for all $G_1$, $G_2\in B_1$ the estimates 
\begin{align*}
 \| D_1 - D_2 \|_{C^0([0,\ell_0])} \leq L_{D,G} \| G_1 - G_2 \|_{\mathcal{L}^\infty([0,L_0])},\quad  \| \beta_1 - \beta_2 \|_{C^0([0,\ell_0])} \leq L_{\beta,G} \| G_1 - G_2 \|_{\mathcal{L}^\infty([0,L_0])}
\end{align*}
hold. Moreover, for all $G\in B_1$ the coefficients $D_G$ and $\beta_G$ are uniformly elliptic and Lipschitz continuous, that is, there exist constants $L_{D,x}$ and $L_{\beta,x}$ such that for all $x_1$, $x_2\in [0,\ell_0]$
\begin{align*}
|D_G(x_1) - D_G(x_2)| \leq L_{D, x} |x_1 - x_2|,\quad |\beta_G(x_1) - \beta_G(x_2)| \leq L_{\beta, x} |x_1 - x_2|.
\end{align*}
and
\begin{align*}
 D_G\in [P_0D_{min}, P_1 D_{max}],\,\beta_G \in \Bsqb{ \frac{\beta_{min}}{P_1}, \frac{\beta_{max}}{P_0} }\,.
\end{align*}
All constants depend only on the global constants. 
\end{lemma}

\begin{proof}
By Lemma~\ref{lem:BoundStress} there exist $P_0$, $P_1\in (0,\infty)$ which depend only on the global constants with $\phi_G'\in [P_0,P_1]$ on $[g(0),g(L_0)]$. For all $x\in [0,\ell_0]$
\begin{align*}
& | (\phi_1'\circ \phi_1^{-1})(x) - (\phi_2'\circ \phi_2^{-1})(x) |  = |(\phi_1' \circ g_1 \circ y_1^{-1})(x) - (\phi_2' \circ g_2 \circ y_2^{-1})(x) | \\ & \qquad \leq | (\phi_1'\circ g_1)(y_1^{-1}(x)) - (\phi_1'\circ g_1)(y_2^{-1}(x)) | + | (\phi_1'\circ g_1)(y_2^{-1}(x)) - (\phi_2'\circ g_2)(y_2^{-1}(x)) |\,.
\end{align*}
For the first term we use Lemma~\ref{lem:LipInX} (iii) and find  
\begin{align*}
 | (\phi_1'\circ g_1)(y_1^{-1}(x)) - (\phi_1'\circ g_1)(y_2^{-1}(x)) | & \leq L_{\phi'\circ g,X} |y_1^{-1} - y_2^{-1} |(x) \leq L_{\phi'\circ g,X} \| y_1^{-1} - y_2^{-1}\|_{C^0([0,\ell_0])}\,.
\end{align*}
By Lemma~\ref{lem:BoundInverseFunction} and since $y' = (\phi'\circ g)g'\in [\Gamma_0 P_0,\Gamma_1 P_1]$ almost everywhere we may define a constant $L_{y^{-1},G}$ such that
\begin{align*}
 \| y_1^{-1} - y_2^{-1}\|_{C^0([0,\ell_0])}\leq (\Gamma_0 P_0)^{-1} L_{y,G} \| G_1- G_2\|_{\mathcal{L}^{\infty}([0,L_0])} = L_{y^{-1},G}\| G_1- G_2\|_{\mathcal{L}^{\infty}([0,L_0])}\,.
\end{align*}
The second term is bounded by $\| \phi_1' \circ g_1 - \phi_2' \circ g_2\|_{C^0([0,L_0])}$ and can be estimated by Lemma~\ref{lem:lipconty}. Both inequalities imply
\begin{align*}
 \sup_{x\in [0,\ell_0]}| (\phi_1'\circ \phi_1^{-1})(x) - (\phi_2'\circ \phi_2^{-1})(x) | \leq ( L_{\phi'\circ g,X} L_{y^{-1},G} + L_{\phi'\circ g}) \| G_1- G_2\|_{\mathcal{L}^{\infty}([0,L_0])}
\end{align*}
and hence Lipschitz continuity in $C^0([0,\ell_0])$. Consequently the map $B_1\to C^0([0,\ell_0])$, $G\mapsto \phi_G'\circ \phi_G^{-1}$ is Lipschitz continuous and the Lipschitz constant depends only on the global constants. By definition, $D_G(x) = (\phi_G'\circ \phi_G^{-1}) \cdot (D_0 \circ y^{-1})$ and the second factor in this formula is estimated by Lemmas~\ref{lem:PropertiesLip} and~\ref{lem:BoundInverseFunction} as follows: For all $x\in [0,\ell_0]$
\begin{align*}
 | ( D_0\circ y_1^{-1}) (x) - (D_0\circ y_2^{-1})(x) | & \leq L_{D_0,X} \|y_1^{-1} - y_2^{-1}\|_{C^0([0,\ell_0])}\\ &\leq L_{D_0,X}L_{y^{-1},G} \| G_1 - G_2 \|_{\mathcal{L}^\infty([0,L_0])}\,.
\end{align*}
Since the map $D:B_1 \to C^0([0,\ell_0])$, $G\mapsto D_G$ is the product of two bounded and Lipschitz continuous functions with the same domain and range and since $C^0([0,\ell_0])$ is a Banach algebra, the proof of the Lipschitz continuity of $D_G$ in $G$ follows from Lemma~\ref{lem:PropertiesLip}.

To prove the Lipschitz continuity of the mapping $G\mapsto \beta_G = [\phi_G'\circ \phi_G^{-1}]^{-1}\beta_0\circ y_G^{-1}$, note that 
\begin{align*}
\frac{1}{\phi_1'\circ \phi_1^{-1}} - \frac{1}{\phi_2'\circ \phi_2^{-1}}  = \frac{ \phi_2'\circ \phi_2^{-1} - \phi_1'\circ \phi_1^{-1} }{(\phi_1'\circ \phi_1^{-1})\cdot (\phi_2'\circ \phi_2^{-1})}\,.
\end{align*}
Since the denominator is bounded from below by $P_0^2$, the map $[\phi'\circ \phi^{-1}]^{-1}:B_1\to C^0([0,\ell_0])$, $G\mapsto [\phi_G' \circ \phi_G^{-1}]^{-1}$ is Lipschitz continuous. The Lipschitz continuity of $\beta_0\circ y^{-1}$ follows as the Lipschitz continuity of the map $D_0\circ y^{-1}$ and thus the Lipschitz continuity of $G\mapsto \beta_G$ has been established.

Finally, $\phi_G'\circ \phi_G^{-1} \in [ P_0,P_1]$ and this estimate implies the uniform ellipticity for $D_G$ and $\beta_G$, respectively. This argument completes the proof of the lemma.
\end{proof}

\begin{lemma}
\label{lem:NisH2}
Suppose that $D_0$, $\beta_0\in W^{1,\infty}([0,L_0])$ satisfy~\eqref{bounds:DandBeta}, that (W1)--(W4) hold and that $B_1=B(G_0,R_1)$ with~\eqref{boundsonG}. Then there exists a constant $M_{n,H^2}$ which depends only on the data of the problem such that the unique weak solution in $H^1(0,\ell_0)$ of the reaction-diffusion equation induced by $G\in B_1$ with 
\begin{align*}
 -(D_G  n_G')' + \beta_G n_G & = 0\quad \text{ in }[0,\ell_0]\,,\\
 n_G(0) & = n_L\,,\\
 n_G(\ell_0) & = n_R\
\end{align*}
satisfies the a~priori bound 
\begin{align*}
 \| n_G \|_{H^2(0,\ell_0)} \leq M_{n,H^2}\,.
\end{align*}
Moreover, if $n_L$, $n_R\geq 0$, then $n_G\geq 0$ on $[0,\ell_0]$.
\end{lemma}

\begin{proof}
Define
\begin{align*}
 \widehat{n}(x) = \frac{n_R-n_L}{\ell_0} x + n_L\,,\quad \widetilde{n}(x) = n_G(x) - \widehat{n}(x)\,. 
\end{align*}
Then $\widetilde{n}(0)=\widetilde{n}(\ell_0) = 0$ and $\widetilde{n}\in H_0^1(0,\ell_0)$ is a weak solution of the equation 
\begin{align*}
 -(D_G \widetilde{n}')' + \beta_G \widetilde{n} = (D_G \widehat{n}')' - \beta_G \widehat{n} = f_G' + h_G 
\end{align*}
with
\begin{align*}
f_G = D_G \widehat{n}'=-D_G\frac{n_R-n_L}{\ell_0},\,\quad h_G = -\beta_G \widehat{n}\,.
\end{align*}
Thus, it suffices to prove that for any weak solution $\widetilde{n}\in H_0^1(0,\ell_0)$ of the equation $-(D_G \widetilde{n}')' + \beta_G \widetilde{n} = f_G'  + h_G$ the a priori estimate $\| \widetilde{n} \|_{H^2(0,\ell_0)}\leq M_{\widetilde{n},H^2}$ holds, where $M_{\widetilde{n},H^2}$ depends only on the global constants. Lemma~\ref{lem:BoundStress} implies the existence of constants $P_0$, $P_1\in (0,\infty)$ which depend only on the global constants such that $\phi_G'\in [P_0,P_1]$ on $[g(0),g(L_0)]$. Therefore the coefficient $D_G$ is uniformly bounded and elliptic with $D_G\in[ D_{min}P_0, D_{max}P_1]$ on $[0,\ell_0]$ and the coefficient $\beta_G$ is uniformly bounded with $\beta_G\in[ \beta_{min}/P_1, \beta_{max}/P_0]$ on $[0,\ell_0]$. The existence of a unique solution in $H^1(0,\ell_0)$ follows with Lax-Milgram. If one uses $\widetilde{n}$ as test function then one finds 
\begin{align*}
P_0 D_{min}\int_0^{\ell_0} (\widetilde{n}')^2 \dv{x} +\frac{\beta_{min}}{P_1}\int_0^{\ell_0} \widetilde{n}^2\dv{x} \leq -\int_0^{\ell_0} f_G \widetilde{n}' \dv{x} + \int_0^{\ell_0} h_G \widetilde{n}\dv{x}
\end{align*}
and Young's inequality leads to the a~priori estimate $\|\widetilde{n}\|_{1,2}\leq C$ where $C$ depends only on the global constants. The $H^2$ regularity follows by elliptic regularity. By Lemmas~\ref{lem:PropertiesReactionDiffusion} the diffusion coefficient $D_G$ is uniformly Lipschitz continuous and the Lipschitz constant depends only on the data. For $\zeta\in C_c^\infty(0,\ell_0)$ and $h\in \R$, $|h|>0$ small enough, one defines the difference quotient $D_h$ and the test function $\psi = D_{-h}(-\zeta^2 D_h \widetilde{n})\in H_0^1(0,\ell_0)$. The leading order term on the left-hand side is given by 
\begin{align*}
\int_0^{\ell_0} D_G(x) \widetilde{n}'(x) (D_{-h}(-\zeta^2 D_h \widetilde{n}(x)))' \dv{x} = \int_0^{\ell_0} D_h(D_G(x) \widetilde{n}'(x)) (\zeta^2 D_h \widetilde{n}(x))' \dv{x}
\end{align*}
and the leading order term on the right-hand side is given by 
\begin{align*}
\int_0^{\ell_0} f_G  (D_{-h}(-\zeta^2 D_h \widetilde{n}))'\dv{x} =  \int_0^{\ell_0} D_{h}f_G  (\zeta^2 D_h \widetilde{n})'\dv{x}\,.
\end{align*}
The first term provides the term $\|\zeta D_h \widetilde{n}'\|_{L^2(0,\ell_0)}^2$. If one rearranges all terms, then one finds an estimate for this term from above by $\| \widetilde{n} \|_{H^1(0,\ell_0)}^2$ and constants that depend only on the global constants, see Lemma~\ref{lem:PropertiesReactionDiffusion} for the properties of $D_G$ and $\beta_G$ which also provide estimates for the Lipschitz constant of $f_G$. Consequently, $\widetilde{n}\in H_{loc}^{2}(0,\ell_0)$, the equation holds pointwise almost everywhere and one can solve for $D_G \widetilde{n}''$ to obtain a global estimate in $H^2(0,\ell_0)$. 

Suppose now that $n_L$, $n_R\geq 0$. To prove the lower bound $n_G\geq 0$ we use the test function $\zeta_G = \min\{ n_G, 0 \}\in H_0^1(0,\ell_0)$. Since $\zeta_G' = \chi_{\{ n_G < 0 \}}  n_G'$ we find 
\begin{align*}
0 & = \int_0^{\ell_0} D_G(x) n_G'(x) \zeta_G'(x) \dv{x} + \int_0^{\ell_0} \beta_G(x) n_G(x) \zeta_G(x)\dv{x}  \\ & = \int_0^{\ell_0} D_G(x) | \zeta_G'(x)|^2 \dv{x} + \int_0^{\ell_0} \beta_G(x) \zeta_G^2(x) \dv{x}\geq 0\,.
\end{align*}
Thus $\zeta_G'=0$ and in view of Poincar\'e's inequality, $\zeta_G=0$, that is, $n_G \geq 0$.
\end{proof}

\begin{lemma}
\label{lem:LipN}
Suppose that $W$ satisfies (W1)--(W4) and that $B_1=B(G_0,R_1)$ with~\eqref{boundsonG}. The map $N:B_1=B(G_0,R_1)\to C^0([0,L_0])$, $G\mapsto N(G)=n_G\circ y_G$, is Lipschitz continuous and there exist constants $M_{N,G}$ and $L_{N,G}$ which only depend on $G_0$ and $R_1$ such that for all $G$, $G_1$, $G_2\in B_1$
 \begin{align*}
  \| N(G) \|_{C^0([0,L_0])} \leq M_{N,G}\,,\quad \| N(G_1) - N(G_2) \|_{C^0([0,L_0])} \leq L_{N,G} \| G_1 - G_2 \|_{\mathcal{L}^\infty([0,L_0]}\,.
 \end{align*}
\end{lemma}

\begin{proof}
Suppose that $G_1$, $G_2\in B_1$ and that $n_1$, $n_2\in W^{1,2}([0,\ell_0])$ are the weak solutions of the reaction-diffusion equations with diffusion constants $D_1$, $D_2$ and absorption rates $\beta_1$, $\beta_2$, respectively. By Lemma~\ref{lem:NisH2}, $n_1$, $n_2\in H^2(0,\ell_0)$ with uniform bound on the $H^2$-norm, independent of $G\in B_1$.

We first verify the Lipschitz continuity of the map $n:\mathcal{L}^\infty([0,L_0])\to H^1(0,\ell_0)$, $G\mapsto n(G)$. Since $n_1 - n_2\in H_0^1(0,\ell_0)$, we may use $n_1 - n_2$ as a test function in the weak formulations for $n_1$ and $n_2$, i.e.,
\begin{align*}
 \int_0^{\ell_0} D_i(x) n_i'(x) (n_1 - n_2)'(x) + \beta_i(x)n_i(x) (n_1 - n_2)(x) \dv{x} = 0\,,\quad i=1,2\,.
\end{align*}
The difference of these two equations is given by 
\begin{align*}
 \int_0^{\ell_0} (D_1 n_1'-D_2 n_2') (n_1' - n_2') + (\beta_1 n_1 - \beta_2 n_2) (n_1 - n_2) \dv{x} = 0 
\end{align*}
and can be rewritten as 
\begin{align*}
& \int_0^{\ell_0} D_1 (n_1'-n_2')^2+ (D_1 - D_2) n_2' (n_1'-n_2') \dv{x} \\ & \quad + \int_{0}^{\ell_0} \beta_1 (n_1-n_2)^2 + (\beta_1 - \beta_2) n_2 (n_1 - n_2) \dv{x} = 0 \,.
\end{align*}
Consequently there exists a constant $C$ which depends only on the global constants with 
\begin{align*}
 \| n_1 - n_2 \|_{H^1(0,\ell_0)} \leq C \|D_1 - D_2\|_{L^\infty(0,\ell_0)} \| n_2'\|_{L^2(0,\ell_0)} + C\|\beta_1 - \beta_2\|_{L^\infty(0,\ell_0)} \| n_2\|_{L^2(0,\ell_0)}\,.
\end{align*}
The Lipschitz continuity of $n_G$ in $G$ with values in $H^1(0,\ell_0)$ follows by Lemma~\ref{lem:PropertiesReactionDiffusion} and we use the Sobolev embedding $H^1(0,\ell_0)\hookrightarrow C^0([0,\ell_0])$ to conclude the Lipschitz continuity in $C^0([0,\ell_0])$. 

To prove the Lipschitz continuity of $N_G$ in $G$ in $C^0([0,L_0])$ we write for $X\in [0,L_0]$
\begin{align*}
|N_1(X) - N_2(X)| & =  |(n_1\circ y_1)(X) - (n_2\circ y_2)(X)| \\ 
 & \leq  |(n_1\circ y_1)(X) - (n_1\circ y_2)(X)| +  |(n_1\circ y_2)(X) - (n_2\circ y_2)(X)|   \\ & \leq \| n_1' \|_{C^0([0,\ell_0])} |y_1(X) - y_2(X) | + \sup_{x\in [0,\ell_0]} |n_1(x) - n_2(x) | \\ & \leq \|n_1' \|_{C^0([0,\ell_0])} \| y_1 - y_2 \|_{C^0([0,L_0])} + \| n_1 - n_2 \|_{C^0([0,\ell_0])} \\ & \leq C \| G_1 - G_2 \|_{\mathcal{L}^\infty([0,L_0])}\,.
\end{align*}
We take the supremum in $X$ and obtain the assertion where we used, in the last estimate, the embedding $H^2(0,\ell_0)\hookrightarrow C^1([0,\ell_0])$ and the uniform bounds in Lemma~\ref{lem:NisH2}.
\end{proof}

\section{Global existence with Banach space methods}\label{sec:existence}

In this section we prove the existence of a solution of the parameter dependent system of ordinary differential equations for $G(X,t)$ based on a formulation of the problem as an ordinary differential equation in the Banach space $\mathcal{L}^\infty([0,L_0])$. The advantage of this approach is that, for a solution $G\in C^1([0,T];\mathcal{L}^\infty([0,L_0])$, the paths $t\mapsto G(X,t)$ are $C^1$ for all $X\in [0,L_0]$ and this fact allows us to identify these paths as solutions of the system of ordinary differential equations depending on the parameter $X$. Our global existence result is based on the following assumptions concerning the existence of comparison functions, where we denote by $\mathcal{L}_+^\infty([0,L_0])\subset\mathcal{L}^\infty([0,L_0])$ the set of all measurable and positive functions, 
\begin{align*}
 \mathcal{L}_+^\infty([0,L_0]) = \nset{ G\in \mathcal{L}^\infty([0,L_0]); \text{ for all }X\in [0,L_0]: G(X)>0 }\,.
\end{align*}
The first assumption (G1) is crucial in order to ensure global existence of solutions since it will allow us to construct global sub- and supersolutions with $G_0(X,t)\leq G(X,t) \leq G_1(X,t)$ for all $X\in [0,L_0]$ and $t\in [0,T]$. This assumption can be interpreted as saying that stress and nutrient concentration have an effect on the growth rate but that the effect is limited in the sense that there is a maximal growth rate for $G>1$ and a maximal absorption rate for $G<1$ which lead to exponential but finite growth on finite time intervals.

We define the autonomous ordinary differential equation 
\begin{align*}
 \dot{G}(t) = \widehat{\mathcal{G}}(G(t))\,,\quad G(0) = 1
\end{align*}
in $\mathcal{L}^\infty([0,L_0])$ where $B_1=B(G_0,R_1)$ with~\eqref{boundsonG},
\begin{align*}
 \widehat{\mathcal{G}} : B_1 \to \mathcal{L}^\infty([0,L_0])\,,\quad \widehat{\mathcal{G}}(G)(X) = \mathcal{G}(G(X), S(G), N(G)(X), X)
\end{align*}
and where the constitutive function $\mathcal{G}$ on the right-hand side satisfies the following conditions:

\begin{itemize}[leftmargin=.4in]
 \item [(G1)] There exist locally Lipschitz continuous functions $\widehat{\mathcal{G}}_0$, $\widehat{\mathcal{G}}_1: \mathcal{L}_+^\infty([0,L_0]) \to \mathcal{L}^\infty([0,L_0])$ such that $\widehat{\mathcal{G}}_i(G)(X) = \mathcal{G}_i(G(X),X)$ for locally Lipschitz continuous functions $\mathcal{G}_i:(0,\infty)\times [0,L_0]\to \R$ such that the ordinary differential equations 
 \begin{align*}
  \dot{G}_i = \widehat{\mathcal{G}}_i(G_i)\,,\quad i=1,\,2\,,\quad G_i(0)=1
 \end{align*}
have global solutions $G_i \in C^1( [0,T]; \mathcal{L}_+^\infty([0,L_0]))$ with 
\begin{align*}
 G_0(X,t) < G_1(X,t) \quad \text{ for all } t\in (0,T],\, X\in [0,L_0]
\end{align*}
and 
\begin{align*}
 G_{min} & = \inf_{[0,L_0]} \inf_{[0,T]} G_0(X,t)>0\,,\\
 G_{max} & = \sup_{[0,L_0]} \sup_{[0,T]} G_1(X,t)<\infty\,.
\end{align*}

\item [(G2)] $\mathcal{G}:\R_{>0}\times \R\times \R_{\geq 0}\times [0,L_0]\to \R$ and for all $X\in [0,L_0]$
 \begin{align*}
  \mathcal{G}(\cdot, \cdot, \cdot, X) : \R_{>0}\times \R\times  \R_{\geq 0}\to \R
 \end{align*}
is continuous and for all $(G, S, N)\in \R_{>0}\times\R\times \R_{\geq 0}$, $\mathcal{G}(G, S, N,\cdot)$ is measurable. 

 \item [(G3)] For all $\Gamma_i$, $\Sigma_i$, $H_i\in \R$, $i=1,2$, with $0< \Gamma_0 < \Gamma_1$, $\Sigma_0 < \Sigma_1$, $0\leq H_0<H_1$ and for all $X\in [0,L_0]$ the function 
\begin{align*}
 \mathcal{G}(\cdot, \cdot, \cdot, X):[\Gamma_0, \Gamma_1]\times [\Sigma_0, \Sigma_1] \times [H_0, H_1]\to \R\,,\quad (G, S, N) \mapsto \mathcal{G}(G, S, N, X)
\end{align*}
is uniformly Lipschitz continuous and bounded with constants $L_{\mathcal{G}}$ and $M_{\mathcal{G}}$ and the constants are independent of $X$.
 
 \item [(G4)] For all $G>0$, $S\in \R$, $N\geq 0$, $X\in [0,L_0]$
 \begin{align*}
  \mathcal{G}_0(G,X) < \mathcal{G}(G, S, N, X) < \mathcal{G}_1(G,X)\,.
 \end{align*}
\end{itemize}

Note, that by Lemma~\ref{lem:NisH2} the nutrients concentration $n_G$ is continuous and nonnegative and the pointwise evaluation of $N_G = n_G \circ y_G$ in the growth dynamics is justified by the Sobolev embedding theorem.

\begin{theorem}\label{thm:ExistenceFullCoupling}
 Suppose that $T>0$, that the free energy density $W$ satisfies (W1)--(W4) and  that the growth rate $\widehat{\mathcal{G}}$ satisfies (G1)--(G4). Then there exists a unique solution $G$ in the sense of Definition~\ref{def:solutionG} on $[0,T]$.
\end{theorem}

\begin{proof}
Global existence will be shown based on Picard-Lindel\"of's theorem on the local existence together with the existence of sub- and supersolutions which follows from (G1). In fact, the sub- and supersolutions guarantee that the solution is contained in a closed ball $B_0$ for $t\in [0,T]$. In view of the lower bound in (G1), there exists a closed ball $B_1$ with the same center and a larger radius which is contained in $\mathcal{L}_+^\infty([0,L_0])$. Therefore, there exists a positive radius $r$ such that for every $G\in B_0$ the closed ball $B(G,r)$ is contained in $B_1$. Since all Lipschitz constants are uniformly bounded on $B_1$, one obtains a uniform bound on the local time of existence $T_1$ in Picard-Lindel\"of's theorem using $G_0\in B_0$ as an initial condition.

\textit{Step~1: Existence of global sub- and supersolutions.} By (G1), the solutions of the ordinary differential equations $\dot{G}_i = \widehat{\mathcal{G}}_i(G_i)$, $G_i(0)=1$, $i=1,2$, exist in the Banach space $C^1([0,T];\mathcal{L}^\infty([0,L_0]))$ and satisfy
\begin{align*}
 G_{min} & = \inf_{[0,L_0]} \inf_{[0,T]} G_0(X,t)\in(0,\infty)\,,\\
 G_{max} & = \sup_{[0,L_0]} \sup_{[0,T]} G_1(X,t)\in (0,\infty)\,.
\end{align*}
Define the closed balls
\begin{align*}
 B_0  = \overline{B_{\mathcal{L}^\infty}}\Barg{ \frac{G_{min}+G_{max}}{2}, \frac{G_{max}-G_{min}}{2} } \,,\quad 
 B_1  = \overline{B_{\mathcal{L}^\infty}}\Barg{ \frac{G_{min}+G_{max}}{2}, \frac{G_{max}}{2} }
\end{align*}
and note that $B_1\subset \mathcal{L}^\infty_+([0,L_0])$. For each $G\in B_0$, the closed ball $\overline{B_{\mathcal{L}^\infty}}( G, G_{min}/2)$ is contained in $B_1$. Assumption (G2) ensures that $\mathcal{G}$ is a Carath\'eodory function. For $G\in B_1$ the function $N(G)\in C^0([0,L_0])$ is continuous and measurable and therefore the function $X\mapsto \mathcal{G}(G(X), S(G), N(G)(X),X)$ is measurable. 

\textit{Step~2: Verification of the assumptions in Picard-Lindel\"of's theorem.} Set $t_0=0$, $G_0=1$. By Lemma~\ref{lem:BoundStress} applied to $B_1$, there exist $P_0$, $P_1$, $\Sigma_0$, $\Sigma_1$ with $P_0>0$, $P_0<P_1$, $\Sigma_0<\Sigma_1$ such that for all $G\in B_1$ the associated stress $S(G)$ and deformation gradient $\phi'(G)$ are uniformly bounded with $S(G)\in [\Sigma_0, \Sigma_1]$ and $\phi'(G)\in [P_0,P_1]$ on $[g(0),g(L_0)]$. Lemma~\ref{lem:LipS} asserts that the map $S:B_1\to \R$, $G\mapsto S(G)$ is Lipschitz continuous and bounded with constants $L_{S,G}$ and $M_{S,G}$, respectively. Lemma~\ref{lem:LipN} implies that $N:B_1 \to C^0([0,L_0])$ is Lipschitz continuous and bounded with constants $L_{N,G}$ and $M_{N,G}$. Moreover, Lemma~\ref{lem:NisH2} guarantees that we may choose $H_0=0$ and a constant $H_1$ which only depends on the data of the problem with $N_G\in [H_0,H_1]$ on $[0,L_0]$. Finally, we may define $\Gamma_0=G_{min}/2>0$ and $\Gamma_1=G_{max}+G_{min}/2$ and with these constants $B_1$ satisfies~\eqref{boundsonG}.

With this choice of $\Gamma_i$, $\Sigma_i$ and $H_i$, $i=1,2$, assumption (G3) guarantees that $\mathcal{G}$ is uniformly Lipschitz continuous in the first three arguments and bounded on $[\Gamma_0,\Gamma_1]\times [\Sigma_0,\Sigma_1]\times [H_0,H_1]$ and the corresponding constants are independent of $X$. The right-hand side of the ordinary differential equation is given by 
\begin{align*}
 \widehat{\mathcal{G}}(G)(X) = \mathcal{G}( G(X), S(G), N(G)(X), X)
\end{align*}
and satisfies for all $G_1$, $G_2\in B_1$,
\begin{align*}
& \| \widehat{\mathcal{G}}(G_1) - \widehat{\mathcal{G}}(G_2) \|_{\mathcal{L}^\infty([0,L_0])} \\ & \quad \leq \sup_{X\in [0,L_0]} | \mathcal{G}(G_1(X), S(G_1), N(G_1)(X), X) - \mathcal{G}(G_2(X), S(G_2), N(G_2)(X), X) | \\ & \quad \leq \sup_{X\in [0,L_0]} L_{\mathcal{G}} \bsqb{ | G_1(X) - G_2(X) | + |S(G_1) - S(G_2)| + |N(G_1)(X) - N(G_2)(X)| } \\ &\quad  \leq L_{\mathcal{G}} (1 + L_{S,G} + L_{N,G}) \| G_1 - G_2 \|_{\mathcal{L}^\infty([0,L_0])}.
\end{align*}
Furthermore, for all $G\in B_1$, 
\begin{align*}
 \| \widehat{\mathcal{G}}(G) \|_{\mathcal{L}^\infty([0,L_0])} & = \sup_{X\in [0,L_0]} | \mathcal{G}(G(X), S(G), N(G)(X), X) | \leq M_{\mathcal{G}}\,.
\end{align*}
Define $K_0 = L_{\mathcal{G}} (1 + L_{S,G} +L_{N,G})$, $M_0 = M_{\mathcal{G}}$. Suppose that $T_1>0$ is small enough such that the inequalities 
\begin{align*}
 K_0 T_1 <1\,,\quad T_1 \leq \min\{ T, \frac{1}{2} G_{min}/M_0 \}
\end{align*}
hold. By Picard-Lindel\"of's theorem in Theorem~\ref{PicardLindeloef} there exists a unique solution $G\in C^1([0,T_1];\mathcal{L}^\infty([0,L_0]))$ of the ordinary differential equation $\dot{G} = \widehat{\mathcal{G}}(G)$.

\textit{Step~3: Uniform a~priori bounds.} For all $X\in [0,L_0]$ the function $G_{;X}:[0,T_1]\to \R$, $t\mapsto G_{;X}(t) = G(X,t)$ satisfies 
\begin{align*}
 \dot{G}_{;X}(t) = \mathcal{G}( G_{;X}(t), S(G(t)), N(G(t))(X), X)\text{ and }G_{;X}(0) = 1\,.
\end{align*}
By Lemmas~\ref{lem:LipS} and~\ref{lem:LipN}, the functions $t\mapsto S(G(t))$ and $t\mapsto N(G(t))$ are continuous since for $t_1$, $t_2\in [0,T_1]$
\begin{align*}
 | S(G(t_1)) - S(G(t_2)) |& \leq L_{S,G} \| G(t_1) - G(t_2) \|_{\mathcal{L}^\infty([0,L_0])}\,, \\  \| N(G(t_1)) - N(G(t_2)) \|_{C^0(0,L_0])}& \leq L_{N,G} \| G(t_1) - G(t_2) \|_{\mathcal{L}^\infty([0,L_0])}
\end{align*}
and the right-hand side converges to zero as $t_1\to t_2$ since $G\in C^1([0,T_1],\mathcal{L}^\infty([0,L_0]))$. Since $\mathcal{G}$ is uniformly Lipschitz continuous, $G_{;X}$ is in fact the unique solution of the scalar ordinary differential equation 
\begin{align*}
 \dot{H}(t) = \widetilde{\mathcal{G}}_X(H, t)\text{ with }\widetilde{\mathcal{G}}_X(H,t) = \mathcal{G}(H, S(G(t)), N(G(t))(X), X)\,,\quad H(0)=1\,,
\end{align*}
where $\widetilde{G}_X$ is Lipschitz continuous in the first argument and continuous as a function on its domain. By (G4), $\mathcal{G}_0(G,X) < \mathcal{G}(G, S(G(t)), N(G(t))(X), X) < \mathcal{G}_1(G,X)$ and by definition the subsolution $G_0(X,t)$ and the supersolution $G_1(X,t)$ satisfy $G_{0;X}(0)=G_{1;X}(0)=G_{;X}(0)=1$. The following argument shows that $G_{0;X}(\cdot)<G_{;X}(\cdot)<G_{1;X}(\cdot)$ for all $X\in [0,L_0]$ on $(0,T_1]$. 

Indeed, since $\dot{G}_{0;X}(0) < \dot{G}_{;X}(0)$, there exists an $\epsilon>0$ such that $G_{0;X}(t) < G_{;X}(t)$ on $(0,\epsilon)$. We show that this is true on $[0,T_1]$.

If this is not the case, there exists a $t_1\in [\epsilon,T_1]$ such that $G_{0;X}(t_1) = G_{;X}(t_1)$ and $G_{0;X}(t)<G_{;X}(t)$ on $(0,t_1)$. Thus,
\begin{align*}
 \dot{G}_{0;X}(t_1) & = \lim_{s\nearrow t_1} \frac{G_{0;X}(t_1) - G_{0;X}(s)}{t_1 - s} = \lim_{s\nearrow t_1} \frac{G_{;X}(t_1) - G_{0;X}(s)}{t_1 - s}  \\ & \geq  \lim_{s\nearrow t_1} \frac{G_{;X}(t_1) - G_{;X}(s)}{t_1 - s}  = \dot{G}_{;X}(t_1)\,.
\end{align*}
However, in view of (G4) and the assumption $G_{0;X}(t_1) = G_{;X}(t_1)$, 
\begin{align*}
 \dot{G}_{0;X}(t_1) & = \mathcal{G}_0(G_{0;X}(t_1), X)  < \mathcal{G}(G_{0;X}(t_1), S(G(t_1)), N(G(t_1))(X), X) \\ & = \mathcal{G}(G_{;X}(t_1), S(G(t_1)), N(G(t_1))(X), X) = \dot{G}_{;X}(t_1)\,,
\end{align*}
a contradiction. With the analogous argument for the supersolution $G_{1;X}$ we conclude that for all $t\in [0,T_1]$ and $X\in [0,L_0]$ 
\begin{align*}
 G_{0;X}(t) < G_{;X}(t) < G_{1;X}(t)
\end{align*}
and hence $G(t) \in B_0$ for all $t\in [0,T_1]$.

\textit{Step~4: Global existence.} By Step~3, $G(T_1)\in B_0$ and we may set $t_0=T_1$, $T_2 = \min\{T, 2T_1\}$ and $G_0=G(T_1)$ and argue as in Step~2 to continue the solution and to obtain the existence of a solution in $C^1([0,T_2], \mathcal{L}^\infty([0,L_0]))$. In view of Step~3, this solution satisfies $G(t) \in B_0$ for all $t$ and after finitely many steps we have $T_k=T$, that is, the global solution has been constructed. Uniqueness of the solution follows from uniqueness of solutions in Picard-Lindel\"of's theorem. 
\end{proof}

\begin{example}\label{ex:Gforfull}
 In Example~\ref{ex:fullcoupling} the dynamics of the growth process is determined from 
\begin{align*}
 \widehat{\mathcal{G}}:\mathcal{L}_+^\infty([0,L_0]) \to \mathcal{L}^\infty([0,L_0])\,,\quad 
 \widehat{\mathcal{G}}(G)(X) = \gamma(X)\mu(S(G))\eta(N(G)(X)) G(X)
\end{align*}
where $S(G)\in \R$ and $N(G)\in H^2(0,L_0)$ denote the stress and the nutrient concentration induced by the growth via $G$. Here, we generalize the assumptions in Example~\ref{ex:fullcoupling} and assume that $\gamma\in \mathcal{L}^\infty([0,L_0];[\gamma_0,\gamma_1])$, and that $\mu\in W^{1,\infty}(\R;[\mu_0,\mu_1])$ and $\eta\in W^{1,\infty}(0,\infty;[\eta_0,\eta_1])$ are increasing, where $0<\gamma_0<\gamma_1$, $\mu_0<\mu_1$, and $0\leq \eta_0\leq \eta_1$. 

In order to obtain the existence of a unique global solution on $[0,T]$ with $T>0$ it remains to verify the assumptions (G1)--(G4). To verify (G1) define for $G\in \mathcal{L}_+^\infty([0,L_0])$
\begin{align*}
 \widehat{\mathcal{G}}_0(G)(X) & = \mathcal{G}_0(G(X),X) =  (\gamma_0 \mu_0\eta_0-1) G(X)\,,\\ \widehat{\mathcal{G}}_1(G)(X) & =\mathcal{G}_1(G(X),X) = (\gamma_1 \mu_1 \eta_1+1) G(X)\,.
\end{align*}
The functions $\mathcal{G}_i:(0,\infty)\times [0,L_0]\to \R$, $i=1,2$, are in fact independent of the second variable $X$ and globally Lipschitz continuous in the first variable $G$ and the Lipschitz constant is independent of $X\in [0,L_0]$. The associated equations $\dot{G}_i = \mathcal{G}_i(G)$, $G_i(0)=1$, $i=1,2$, are systems with pure growth that have global solutions as in Example~\ref{ex:puregrowth}. Since $G_i(0)=1$, the solutions are positive and the assumptions on the global lower and upper bounds in (G1) are satisfied. Moreover, the explicit formulas for the solutions verify the strict inequality $G_0<G_1$ on $(0,T]\times [0,L_0]$.  

Concerning (G2), for all $X$ fixed the function $(G,S,N)\mapsto \mathcal{G}(G, S, N, X) = \gamma(X) \mu(S) \eta(N) G$ is continuous in view of the assumptions concerning $\mu$ and $\eta$ and for all $(G, S, N)$ the function $X\mapsto \mathcal{G}(G, S, N, X) = \gamma(X)\mu(S)\eta(N)G$ is measurable since $\gamma$ is measurable. 

To verify (G3), we fix $\Gamma_i$, $\Sigma_i$, $H_i$ with the properties in (G3) and $X\in [0,L_0]$. Then, as a function on $[\Gamma_0,\Gamma_1]\times [\Sigma_0,\Sigma_1]\times [H_0,H_1]$, 
\begin{align*}
 (G,S,N) \mapsto \mathcal{G}(G, S, N, X) = \gamma(X) \mu(S) \eta(N) G
\end{align*}
is Lipschitz continuous since $\mu$ and $\eta$ are Lipschitz continuous and bounded and the function is bounded since $\gamma$, $\mu$, and $\eta$ are bounded. Moreover, the corresponding constants $L_{\mathcal{G}}$ and $M_{\mathcal{G}}$ are independent of $X$.

Finally, (G4) is satisfied since for $G>0$
\begin{align*}
 \mathcal{G}_0(G) = (\gamma_0\mu_0\eta_0-1) G < \gamma_0 \mu_0\eta_0 G \leq \gamma(X) \mu(S) \eta(N) G = \mathcal{G}(G, S, N, X)
\end{align*}
and the calculation for the upper bound is analogous. Consequently, the system in Example~\ref{ex:fullcoupling} has a unique global solution for any stored energy density $W$ which satisfies (W1)--(W4).
\end{example}

\section{Conclusions}\label{sec:conclusions}

In this article, the focus of the analysis is a general framework that identifies abstract conditions on the free energy density $W$ and the growth dynamics $\mathcal{G}$ that lead to a well-posed model with global solutions. In~\cite{Bangert2022}, initially motivated by a view towards numerical schemes, very specific assumptions were made concerning $W$ and $\mathcal{G}$. In particular the cases of materials with two parts as in Example~\ref{ex:puregrowth} or, more generally, with finitely many parts were discussed in great detail. It is noteworthy, that these cases are not included in the general theory presented in this article since the assumptions $W$, $W_p\in C^0([0,L_0]\times (0,\infty))$ in (W1) are not satisfied. Here it is important to note that these assumptions are not needed in order to obtain the Euler-Lagrange equation. In fact, the existence of a minimizer in $W^{1,1}$ and the validity of the Euler-Lagrange equation for almost all $X\in [0,L_0]$ follow if $\partial_p W$ is merely a Carath\'eodory function. In the special case of coefficients that are piecewise constant the solution $\phi$ is continuous on $[g(0),g(L_0)]$ and piecewise affine and continuous dependence of the stress and the nutrients can be verified by explicit calculations since~\eqref{reactiondiffusion} is a boundary value problem for an ordinary differential equation with constant coefficients which can be solved explicitly. We sketch the arguments in Appendix~\ref{sec:appendix}.

If the growth dynamics depend only on the stress but not on a concentration of nutrients, the assumptions in Theorem~\ref{thm:ExistenceFullCoupling} can be weakened as well. In fact, the property that $\phi_G' \circ g$ is continuous in the proof of Lemma~\ref{lem:LipS} can be replaced by measurability which can be obtained appealing to measurability of implicitly defined functions, see, e.g., \cite[Corollary~18.8]{AliprantisBorderBook2006}. 

Further generalizations concern the dependence of the growth dynamics on $N$. For example, the function $\eta$ in Example~\ref{ex:fullcoupling} is assumed to be nonnegative and therefore the choice of $\eta(N)=0$ if $N<N_c$ for some critical value $N_c>0$ is included leading to a necrotic core in which no further growth or absorption happens. In particular in view of numerical schemes or in the spirit of the situation with two materials in Example~\ref{ex:puregrowth}, the pointwise evaluation of $\eta$ can be replaced by local averages leading to piecewise constant approximations of the concentration of the nutrients.

Concerning the general assumptions on the growth dynamics, the assumptions (G1) and (G4) are necessary in order to obtain global existence on $[0,T]$. In a certain sense, as illustrated in Example~\ref{ex:Gforfull}, the assumptions imply linear growth of $\mathcal{G}$ in $G$, a fact that is expected for global existence. To illustrate that global existence may fail without the presence of linear bounds consider the situation in Example~\ref{ex:stresscoupling} in which coupling is given only through the stress in the system and $\mathcal{G}(G) = \mu(S(G)) G^2$. Suppose that the material is homogeneous, $W(X,p) = W(p)$, and that $G(t)$ exists for $t\in (0,T)$ for some $T>0$. Since $G(t)$ is independent of $X$, $g(X,t) = G(t)X$ and since $W$ is homogeneous and strictly convex, the minimizer $\phi_G$ is affine, $\phi_G(z,t) = \ell_0 z/ (G(t)  L_0)$. The stress is constant and given by $\partial_pW(\phi'(z)) = \partial_p W(\ell_0/G(t)/ L_0)$. The growth dynamics are therefore given by 
\begin{align*}
 \dot{G}(t) = \mu(\partial_p W(\ell_0/G(t)/L_0) G(t)\,.
\end{align*}
By assumption, $W_p$ is strictly in creasing and invertible and we may choose for some exponent $\delta>0$ the increasing function $\mu(S) = (\partial_p W^{-1}(S))^\delta$ which leads to the equation $\dot{G}(t) = (\ell/ L_0)^\delta G(t)^{2-\delta}$. Depending on the value of $\delta$ one obtains sublinear or superlinear growth and, in the latter case, only local existence of solutions. 

The true challenge remains to address the higher dimensional case. Our one-dimensional analysis uses the fact that every $L^\infty$ function has an anti-derivative and that the natural configuration can be represented by an interval in a critical way. So far, explicit examples in two dimensions have been given in special geometries, for example in polar coordinates. This approach is very promising. However, one of the major open problems is related to the question of whether a minimizer of a nonlinear variational problem is radial. The answer to this question will depend on the boundary conditions and on the state of the stress since compression may lead to buckling effects. A first approach could be to follow the lines of~\cite{KruzikMelchingStefanelliESAIM2020DislocationFreePlasticity} and consider compatible growth. 

\bigskip 

\textbf{Acknowledgment:} The first author was supported by the DFG Research Training Group 2339 Interfaces, Complex Structures, and Singular Limits.

\bigskip
  
\begin{appendix}
\section{The case of piecewise homogeneous materials}\label{sec:appendix}

As in Example~\ref{ex:puregrowth} we suppose that $\Omega = [0,L_0]$, $X_I\in (0,L_0)$ and that the material is homogeneous on $[0,X_I]$ and $(X_I,L_0]$, respectively. The following arguments can be extended to the case that the material is homogeneous on $k$ distinct subintervals that consitute the material body in its reference configuration. In this situation, the growth equation
\begin{align*}
\dot{G}(t) = \widehat{\mathcal{G}}(G(t)),\,G(0)=1\text{ with }\widehat{\mathcal{G}}(G)(X) = \mathcal{G}(G(X), S(G), N(G)(X), X)
\end{align*}
reduces to two ordinary differential equations with $\mathcal{G}(G,S,N,\cdot) = \mathcal{G}(G,S,N,0)$ on $[0,X_I]$ and $\mathcal{G}(G,S,N,\cdot) = \mathcal{G}(G,S,N,L_0)$ on $(X_I,L_0]$. In particular, $G(X,t) = \chi_{[0,X_I]}(X)G(0,t) + \chi_{(X_I, L_0]}(X) G(L_0,t)$ and we may identify the growth tensor $G(\cdot, t)$ with the two values $G(0,t)$ and $G(L_0,t)$. In order to draw parallels to the general case treated before, we define 
\begin{align*}
 \mathcal{S}^0 & = \{ G=G_0\chi_{[0,X_i]} + G_1 \chi_{(X_I,L_0]},G_0, G_1\in \R\} \\  \mathcal{S}^1 & = \{g\in W^{1,\infty}(0,L_0), g'=G_0\chi_{[0,X_i]} + G_1 \chi_{(X_I,L_0]},G_0, G_1\in \R\}
\end{align*}
and identify all elements in $\mathcal{S}^1$ with the continuous representative. Consequently the ordinary differential equation can be formulated in the Banach space $(\mathcal{S}^0,\|\cdot\|_\infty)$ with $\widehat{\mathcal{G}}:\mathcal{S}^0\to \mathcal{S}^0$ and we require for all $G\in \mathcal{S}^0$ that $\widehat{\mathcal{G}}(G)\in \mathcal{S}^0$. Moreover, for $G_0$, $G_1>0$ the potential $g(\cdot, t)$ with $\partial_X g(X,t)=G(X,t)$ is invertible and bi-Lipschitz.

\begin{example}
In analogy to Example~\ref{ex:fullcoupling} we define 
\begin{align*}
 \widehat{\mathcal{G}}(G)(X) = \gamma(X) \mu(S(G)) \eta(N(G)(X)) G(X)
\end{align*}
and note that this expression defines an element in $\mathcal{S}^0$ if we choose $\gamma = \gamma_0\chi_{[0,X_I)} + \gamma_1 \chi_{(X_I,L_0]}$ with $\gamma_0$, $\gamma_1\in\R$ and if we use a suitable definition for $\eta(N(G)(X))$. As an example we consider the choice of $N$ as a pull-back of a local average of the density of the nutrients $n\in H^1(0,\ell_0)$ in the current configuration $(0,\ell_0)$ given by~\eqref{reactiondiffusion} in the discussion below.
\end{example}

In view of Picard-Lindel\"of's theorem stated in Theorem~\ref{PicardLindeloef} it is crucial to prove that the right-hand side of the ordinary differential equation depends on the function $G\in \mathcal{S}^0$ in a Lipschitz continuous way and that all estimates depend only on global constants. We identify $G\simeq (G_0,G_1)\in [\Gamma_0, \Gamma_1]^2$. Suppose that $G\in \mathcal{S}^0$ is given and that $g\in \mathcal{S}^1$ with 
\begin{align*}
 g(X) = \left\{\begin{array}{cl} XG_0 & \text{ if } X\in [0,X_i]\,,\\ X_I G_0 + (X-X_I)G_1 & \text{ if }X\in (X_I, L_0]\,. \end{array} \right.
\end{align*}
Define $z_I = g(X_I)$ and note that the induced stored energy density $W_G$ is homogeneous on $[g(0),z_I]$ and $(z_I,g(L_0)]$ with 
\begin{align*}
 W_G(z,p) = \kappa_0 W_0(p) \chi_{[g(0),z_I]}(z) + \kappa_1 W_0(p) \chi_{(z_I,g(L_0)]}(z)\,.
\end{align*}
Since $W_0$ is strict convex, $\phi_G$ is piecewise affine and the four constants $A_0$, $A_1$, $B_0$, $B_1\in \R$ in the expression 
\begin{align*}
\phi_G(z) = (A_0 + B_0 z) \chi_{[g(0),z_I]}(z) + (A_1 + B_1 z) \chi_{(z_I,g(L_0)]}(z)
\end{align*}
are determined from the boundary condition, the continuity of $\phi_G$ in $z_I$ and the fact that the stress in the system is constant. More explicitly, these four conditions lead to the nonlinear system with four unknows and the four equations
\begin{alignat*}{2}
\Phi^1= A_0 + B_0 g(0) & = 0\,, & \qquad\quad  \Phi^3=A_0 + B_0 z_I - (A_1 + B_1 z_I) & = 0 \,,\\ \Phi^2=A_1 + B_1 g(L_0) - \ell_0 & = 0\,, & \Phi^4=\kappa_0 \partial_p W_0(B_0) - \kappa_1 \partial_p W_0(B_1)& =0\,.
\end{alignat*}
Corollary~\ref{cor:ExistenceIG} guarantees the existence of a unique solution $A_0$, $A_1$, $B_0$, $B_1$ for each choice of $G\simeq(G_0,G_1)\in [\Gamma_0, \Gamma_1]^2$. In order to prove that these coefficients are continuously differentiable in $G$ we interpret a solution of the system as a zero of the map $\Phi:(\Gamma_0/2, 2\Gamma_1)^2\times \R^4 \to \R^4$, where $g(0)=0$, $g(L_0)$ and $z_I$ are smooth functions in $G_0$ and $G_1$. Thus, the map $(G_0, G_1, A_0, A_1, B_0, B_1)\mapsto (\Phi^i(G_0,G_1,A_0,A_1,B_0,B_1))_{i=1,\ldots,4}$ is continuously differentiable with $\partial \Phi/\partial(A_0,A_1,B_0,B_1)$ invertible, a fact that can be verified by expanding the determinant along the second column in the matrix in view of $g(0)=0$ and $\partial_{pp} W_0>0$. The theorem on implicit functions shows that the maps $A_i$, $B_i:(\Gamma_0/2,2\Gamma_1)^2\to \R^2$, $i=1,2$, are continuously differentiable and the explicit formula shows that the deriviatives are uniformly bounded.

The foregoing calculations prove that the map $S:[\Gamma_0,\Gamma_1]^2\to \R$, $G\mapsto S(G)$ is $C^1$ and thus globally Lipschitz continuous since $S=\kappa_0\partial_p W_0(\phi_G'(\cdot))|_{[0,z_I]}$ and all functions of $G_0$ and $G_1$ appearing in this expression are $C^1$ with uniform bounds.

\begin{figure}[t]
\centering
\begin{tikzpicture}[scale = 1.8]
\draw (0,0)--(2,0);
\draw (0,0)--(0,2);
\foreach \g/\y in {{G_0}/0.5,{G_1}/1.5} \draw(-0.05,\y) -- (0.05,\y) node[left]{$\g\,\,$};
\draw(-0.05,1) -- (0.05,1) node[left]{$1\,\,$};
\foreach \x/\y in {0/0 , 1/{X_I}, 2/{L_0}} {
  \draw (\x,-0.05)--(\x,0.05) ;
  \draw(\x,-0.1) node[below]{$\y$};
}
\draw[blue, line width=1.2] (0,0.5)--(1,0.5);
\draw[blue, line width=1.2] (1,1.5)--(2,1.5);
\end{tikzpicture}
\hfil
\begin{tikzpicture}[scale = 3.0]
\draw (0,0) -- (1.82,0) node[right]{$z$};
\draw (0,0) -- (0,1.2) node[left]{$\phi_G$};
\foreach \x/\y in {0/0 , 0.72/{z_I}, 1.82/{g(L_0)}} {
  \draw (\x,-0.05)--(\x,0.05) ;
  \draw(\x,-0.1) node[below]{$\y$};
}
\draw[blue, line width=1.2] (0,0) -- (0.72,0.54) -- (1.82,1);
\foreach \g/\y in {{0}/0.0,{\ell_0}/1.0} \draw(-0.03,\y) -- (0.03,\y) node[left]{$\g\,\,$};
\end{tikzpicture}
\caption{Left panel: The reference configuration and the graph of $G$. Values $G<1$ correspond to absorption, values $G>1$ to growth. Right panel: The elastic deformation $\phi_G$ with $L_0=1$, $X_I=0.8$, $G=(0.9,5.5)$, $z_I=0.72$, $g(L_0)=1.82$, $\ell_0=1$, $x_I=0.534$.\label{fig:RefVirt}}
\end{figure}
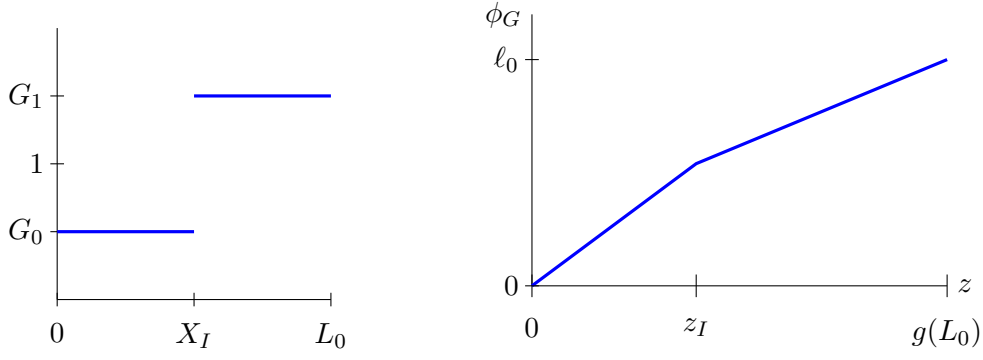

The equation for the nutrients has to be solved on the currert configuration $[0,\ell_0]$ and the coefficients are given by 
\begin{align*}
 D_G(x) & = D_0|_{[0,X_I]}\phi_G'|_{[0,z_I]}\chi_{[0,x_I]}(x) + D_0|_{[X_I,L_0]}\phi_G'|_{(z_I,g(L_0)]}\chi_{(x_I,\ell_0]}(x)\,, \\  \beta_G(x) & =\frac{ \beta_0|_{[0,X_I]}}{\phi_G'|_{[0,z_I]}}\chi_{[0,x_I]}(x) + \frac{ \beta_0|_{[X_I,L_0]}}{\phi_G'|_{(z_I,g(L_0)]}}\chi_{(x_I,\ell_0]}(x)
\end{align*}
where $x_I = \phi(z_I)$ is a $C^1$ function of $G$. The coefficients in the equation $-(D_Gn_G')' + \beta_G n_G=0$ are constant on $[0,x_I]$ and $(x_I, \ell_0]$ and the solution $n_G$ can be interpreted on each subinterval as a solution of a linear ordinary differential equation of second order of the form $-Dn'' + \beta n=0$ with the fundamental solution $n(x) = c^+\exp(\lambda x) + c^- \exp(-\lambda x)$ and $\lambda = \sqrt{\beta/D}$. The coefficients of $n_G$ in the representation 
\begin{align*}
 n_G(x) = ( c_0^+\exp(\lambda_0 x) + c_0^- \exp(-\lambda_0 x) ) \chi_{[0,x_I]} + ( c_1^+\exp(\lambda_1 x) + c_1^- \exp(-\lambda_1 x) ) \chi_{(x_I,\ell_0]}
\end{align*}
are determined from the boundary conditions, the continuity of the nutrients and their flux by 
\begin{align*}
 \Psi^1 & = c_0^+ + c_0^- - n_L = 0\,,\\ \Psi^2 & = c_1^+ \exp(\lambda_1 \ell_0) + c_1^-\exp(-\lambda_1 \ell_0) - n_R =0\,,\\ \Psi^3 & = c_0^+ \exp(\lambda_0 x_I) + c_0^-\exp(-\lambda_0 x_I) - ( c_1^+ \exp(\lambda_1 x_I) + c_1^-\exp(-\lambda_1 x_I) ) = 0\,, \\  \Psi^4 & =D_0|_{[0,X_I]} \phi_G'|_{[0,z_I]}(\lambda_0 c_0^+ \exp(\lambda_0 x_I) -\lambda_0 c_0^-\exp(-\lambda_0 x_I)) \\ & \qquad  -D_0|_{(X_I,L_0]} \phi_G'|_{(z_I,g(L_0)]}( \lambda_1 c_1^+ \exp(\lambda_1 x_I) -\lambda_1 c_1^-\exp(-\lambda_1 x_I) ) = 0 \,.
\end{align*}
Once $\phi_G$ has been obtained, the equation for $n_G$ is determined and the coefficients $D_G$ and $\beta_G$ are piecewise constant, uniformly elliptic and bounded for $G\in [\Gamma_0,\Gamma_1]^2$. Existence of a weak solution and regularity on the subintervals $(0,z_I)$ and $(z_I,\ell_0)$ follows with Lax-Milgram's theorem and elliptic regularity. Consequently, for each $G$ there exists a solution $n_G$ which is of the foregoing form that satisfies the nonlinear equation $\Psi(G_0,G_1,c_0^+,c_0^-,c_1^+,c_1^-)=0$. Since $x_I$, $D_i$, $\beta_i$, $i=1,2$, and the coefficients in the representation of $\phi_G$ are $C^1$ functions of $G$, $\Psi:(\Gamma_0/2,2\Gamma_1)^2\times \R^4\to \R^4$ is of class $C^1$ and the theorem on implicit functions implies that the coefficients $c_i^{\pm}$, $i=1,2$, are $C^1$ functions of $G$. In fact the derivative of $\Psi$ with respect to $c_i^{\pm}$, $i=1,2$, is given
with 
\begin{align*}
 D_{0L} = D_0|_{[0,X_I]},\, D_{0R} = D_0|_{(X_I,L_0]},\, A= \frac{D_{0R}\phi_{GR}'}{D_{0L}\phi_{GL}'}  \cdot \frac{\lambda_1}{\lambda_0}
\end{align*}
after a scaling in the last row by multiplication by $\diag(1,1,1,(D_{0L}\phi_{GL}'\lambda_0)^{-1})$ from the left
\begin{align*}
 \left(\begin{array}{cccc} 1 & 1 & 0 & 0 \\ 0 & 0 & \exp(\lambda_1\ell_0) & \exp(-\lambda_1\ell_0) \\ \exp(\lambda_0 x_I) & \exp(-\lambda_0 x_I)& -\exp(\lambda_1 x_I)& -\exp(-\lambda_1 x_I)\\ \exp(\lambda_0 x_I) & -\exp(-\lambda_0 x_I) & -A\exp(\lambda_1 x_I) & A\exp(-\lambda_1 x_I)\end{array} \right)\,.
\end{align*}
Since we only need to verify that the matrix is invertible, we may subtract the third row from the fourth row and expand the determinant along the first column. The determinant is given by 
\begin{align*}
& \exp(-\ell_0 \lambda_1 - x_I (\lambda_0+\lambda_1)) \\& \quad \cdot\bsqb{ (1+A) (\exp(2x_I\lambda_0 + 2\ell_0 \lambda_1) - \exp(2x_I \lambda_1)) + (-1+A)( \exp(2x_I(\lambda_0+\lambda_1) - \exp(2\ell_0 \lambda_1)) } 
\end{align*}
Since $0<x_I<\ell_0$ this expression is positive and the theorem on implicit functions applicable.
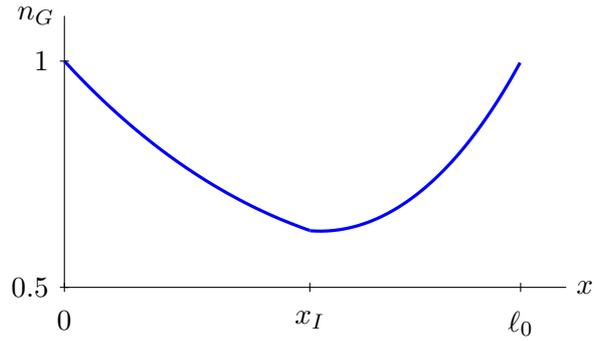
\begin{figure}[t]
\centering
\begin{tikzpicture}[xscale = 6, yscale=6]
\draw (0,0)--(1.1,0) node[right]{$x$};
\draw (0,0)--(0,0.6) node[left]{$n_G$};
\draw[domain=0:0.538,samples=100, blue, line width=1.2] plot ({\x}, {0.0885*exp(1.3368*\x) + 0.9114* exp(-1.3368*\x) -0.5});
\draw[domain=0.538:1, samples=100, blue, line width=1.2] plot ({\x}, {0.082*exp(2.3839*\x) + 1.1884* exp(-2.3839*\x)-0.5});
\foreach \g/\y in {1/0.5, 0.5/0} \draw(-0.01,\y) -- (0.01,\y) node[left]{$\g\,\,$};
\foreach \x/\y in {0/0 , 0.538/{x_I}, 1/{\ell_0}} {
  \draw (\x,-0.01)--(\x,0.01) ;
  \draw(\x,-0.03) node[below]{$\y$};
}
\end{tikzpicture}
\caption{The solution of the diffusion equation with $x_I=0.538$, $D_{0L}=1$, $D_{0R}=8$, $\beta_{0L}=1$, $\beta_{0R}=8$, $\phi_{GL}'=0.72$, $\phi_{GR}'=0.41$, $n_L=1$, $n_R=1$.\label{fig:Current}} 
\end{figure}

It remains to define $N(G)\in \mathcal{S}^0$ in the ordinary differential equation for the growth dynamics. One possible choice is 
\begin{align*}
 N(G)(X) =\frac{1}{x_I} \int_{0}^{x_I} n(x) \dv{x}\cdot\chi_{[0,X_I]}(X) + \frac{1}{\ell_0-x_I} \int_{x_I}^{\ell_0} n(x) \dv{x}\cdot\chi_{(X_I,L_0]}(X)\,.
\end{align*}
Both coefficients in the definition of $N(G)\in \mathcal{S}^0$ are continuously differentiable functions of $G$ and thus the map $N:(\Gamma_0/2, 2\Gamma_1)^2\to \mathcal{S}^0$ is Lipschitz continuous. 
\end{appendix}

\end{document}